\documentclass{article}

\usepackage[utf8]{inputenc} 
\usepackage[T1]{fontenc}    
\usepackage{amsmath,amssymb,amsfonts}
\usepackage{amsthm}
\usepackage{hyperref}       
\usepackage{url}            
\usepackage{booktabs}       
\usepackage{amsfonts}       
\usepackage{nicefrac}       
\usepackage{microtype}      
\usepackage{geometry}
\usepackage{xcolor}         

\usepackage{mathtools}
\mathtoolsset{showonlyrefs}
\usepackage{bm}
\usepackage{subcaption}
\usepackage{algorithm}
\usepackage{algpseudocode}
\usepackage{comment}

\newcommand{\R}{\mathbb{R}}
\newcommand{\N}{\mathbb{N}}
\newcommand{\Z}{\mathbb{Z}}
\renewcommand{\d}{\mathrm{d}}

\newcommand{\E}{\mathbb E}
\newcommand{\Sp}{\mathbb S}

\newcommand{\Beta}{\mathrm{B}}
\newcommand\dx{\mathrm{d}}
\newcommand\adj{\mathrm{H}}

\newcommand{\tT}{\mathrm{T}}

\newtheorem{theorem}{Theorem}[section]
\newtheorem{lemma}[theorem]{Lemma}
\newtheorem{proposition}[theorem]{Proposition}
\newtheorem{remark}[theorem]{Remark}
\newtheorem{corollary}[theorem]{Corollary}
\theoremstyle{definition}
\newtheorem{example}[theorem]{Example}

\title{Fast Kernel Summation in High Dimensions via Slicing and Fourier Transforms}

\author{Johannes Hertrich\thanks{University College London,
\texttt{j.hertrich@ucl.ac.uk}.}
}

\usepackage{amsopn}

\begin{document}

\maketitle

\begin{abstract}
Kernel-based methods are heavily used in machine learning.
However, they suffer from $O(N^2)$ complexity in the number $N$ of considered data points.
In this paper, we propose an approximation procedure, which reduces this complexity to $O(N)$.
Our approach is based on two ideas.
First, we prove that any radial kernel with analytic basis function can be represented as sliced version of some one-dimensional kernel and derive an analytic formula for the one-dimensional counterpart.
It turns out that the relation between one- and $d$-dimensional kernels is given by a generalized Riemann-Liouville fractional integral.
Hence, we can reduce the $d$-dimensional kernel summation to a one-dimensional setting.
Second, for solving these one-dimensional problems efficiently, we apply fast Fourier summations on non-equispaced data, a sorting algorithm or a combination of both.
Due to its practical importance we pay special attention to the Gaussian kernel, where we show a dimension-independent error bound and represent its one-dimensional counterpart via a closed-form Fourier transform.
We provide a run time comparison and error estimate of our fast kernel summations.
\end{abstract}

\section{Introduction}

Kernel-based methods are a powerful tool in machine learning. 
Applications include kernel density estimation \cite{P1962,R1956}, classification via support vector machines \cite{SC2008}, dimensionality reduction with kernelized principal component analysis \cite{SS2002,SC2004}, distance measures on the space on probability measures like maximum mean discrepancies or the energy distance \cite{GBRSS2006,szekely2002}, corresponding gradient flows \cite{AKSG2019,GBG2024,HWAH2023,kolouri22}, and methods for Bayesian inference like Stein variational gradient descent \cite{LQ2016}.
All of these methods have in common that they need to evaluate the kernel sums
$
s_m=\sum_{n=1}^Nw_nK(x_n,y_m)
$
for $m=1,...,M$, $x_1,...,x_N,y_1,...,y_M\in\R^d$ and weights $w_1,...,w_N\in\R$.
Evaluating these sums in a straightforward way has a computational complexity $O(MN)$, which coincides with $O(N^2)$ for $M=N$. For large $M$ and $N$, this leads to an intractable computational effort. Moreover, some of these applications consider image datasets which are very high-dimensional.

In this paper, we focus on radial kernels given by $K(x,y)=F(\|x-y\|)$ for some basis function $F\colon\R_{\geq0}\to\R$.
We  improve the computational complexity towards $O(N+M)$ 
and accelerate the numerical computations
by two steps. First, we reduce the $d$ dimensional summation to several one-dimensional ones via slicing. Second, we compute the resulting one dimensional summations in an efficient way 
by non-equispaced Fourier transforms, sorting or a combination of both depending on the kernel.

In a low-dimensional setting, say $d \le 3$, there exist several approaches in the literature to accelerate the summations. We provide a (non exhaustive) list in the ``related work'' paragraph below. 
For smooth kernels, we mainly focus on fast Fourier summations \cite{KPS2006,PSN2004} which are based on non-equispaced Fourier transforms \cite{B1995,DR1993,PST2001}. 
That is, we expand the function $\phi\coloneqq F(\|\cdot\|)$ into its Fourier series and truncate it accordingly.
Then, the computation of the sums $s_m$ reduces to applying two Fourier transforms on non-equispaced points.
Assuming that the number of considered Fourier coefficients is constant, this leads to a computational complexity of $O(M+N)$. Depending on the required number of Fourier coefficients, we either use discrete (NDFT) or fast Fourier transforms (NFFT) \cite{B1995,DR1993,PST2001} for non-equispaced data.
We refer to Section~\ref{sec:fast_kernel_summation} for a more detailed explanation.
Additionally, we consider certain non-smooth kernels like the negative distance kernel and the Laplacian kernel, where we apply a sorting algorithm similar to \cite{HWAH2023} and a combination of sorting and fast Fourier summations.
Note that the computational effort of the Fourier transform scales exponentially with the dimension such that these approaches are only tractable for $d\leq 3$. 

Therefore, we reduce the $d$-dimensional problem in our first step to the one-dimensional case 
by projecting the points of interest onto lines in directions $\xi$ which are uniformly distributed on the 
unit sphere $\Sp^{d-1} \subset \R^d$.
That is, we assume that there exists a one-dimensional kernel 
$\mathrm{k}\colon\R\times\R\to\R$ 
such that the kernel 
$K\colon\R^d\times\R^d\to\R$ 
is of the form
\begin{equation}\label{eq:sliced_kernel_intro}
K(x,y)=\E_{\xi\sim \mathcal U_{\Sp^{d-1}}}[{\mathrm k}(P_\xi(x),P_\xi(y))],\quad P_\xi(x)=\langle \xi,x\rangle.
\end{equation}
Such kernels appeared in the context of neural gradient flows of sliced maximum mean discrepancies in \cite{HHABCS2023,HWAH2023,kolouri22}.
Interestingly, for radial kernels ${\mathrm k}(x,y) = f (\|x-y\|)$ with basis functions $f$ in a certain function space, 
the kernels $K$ defined by \eqref{eq:sliced_kernel_intro} are again  radial kernels
and the mapping from $f \mapsto F$ is a special case of the generalized Riemann-Liouville (or Weyl) fractional integral transform \cite{L1971,SKM1993}.
While this transform is injective, it maps only onto a smaller function space. 
Therefore not every radial kernel $K$  admits a one-dimensional radial counterpart 
such that \eqref{eq:sliced_kernel_intro} holds true.
We examine the relation for certain kernels appearing in applications.
In particular, we prove that for radial kernels with power-series-like basis function $F$, including Gaussian, Laplacian or Mat\'ern kernels,
there always exists a radial kernel $\mathrm{k}$ 
and provide its basis function analytically in a power-series-like form.
For the important case of the Gaussian kernel, it turns out that $\mathrm{k}$ can be expressed by the confluent hypergeometric function. 
In particular, we have access to a closed-form expression of the Fourier transform making the one-dimensional computations very efficient.
Now, the sums $s_m$ read as
\begin{equation*}
s_m=\E_{\xi\sim \mathcal U_{\Sp^{d-1}}}\Big[\sum_{n=1}^Nw_n\mathrm{k}(P_\xi(x_n),P_\xi(y_m))\Big].
\end{equation*}
By replacing the expectation over $\xi$ by finite sums, 
this reduces the computation of the $d$-dimensional kernel sums 
to the computation of a finite number of one-dimensional kernel sums.
Here, we give an explicit bound of the error which is introduced by discretizing the expectation value.
In particular, we observe that this bound is independent of the dimension 
for some important kernels like Gaussian, Laplacian and Mat\'ern kernels.

We demonstrate the efficiency of our fast kernel summations by numerical experiments including run time comparisons and error estimates for the Gaussian, Laplacian, negative distance (also called energy kernel) and Mat\'ern kernel. 
Moreover, we compare our slicing approach with random Fourier features (RFF) \cite{RR2007,SS2015}. In contrast to the proposed slicing method, RFF are limited to positive definite kernels. This excludes some applications where only conditionally positive definite kernels are used. 
For more details, we refer to Section~\ref{sec:cmp_RFF}.

\paragraph{Outline} In Section~\ref{sec:sliced_kernels}, we study radial kernels and their sliced versions. 
We prove that the relation between their basis functions $F$ and $f$ is given by a certain 
Riemann-Liouville  integral transform and we provide analytical expressions for $f$ 
if $F$ is an analytic function. Special emphasis is paid to the Gaussian kernel. 
Further, we give an error estimate when replacing the expectation value by an empirical one. Interestingly,
for various important kernels, this error does not depend on the dimension of the original summation problem.
In Section~\ref{sec:fast_kernel_summation}, we deal with the fast summation in the one-dimensional case.
For the  distance kernels, the 1d summation can be solved just by a sorting algorithm.
For analytic kernels like the Gaussian, we use our findings about sliced kernels 
to reduce the computation to the one-dimensional Fourier domain and apply NFFTs. 
For kernels as the Laplacian, we can combine sorting  and  NFFTs after a suitable kernel decomposition.
In Section~\ref{sec:numerics}, we verify our results numerically and compare to random Fourier features. 
Finally, conclusions are drawn in Section~\ref{sec:concl}.

\subsection*{Related Work}

\paragraph{Slicing} The idea to reduce high-dimensional problems to the one-dimensional case appeared first for the Wasserstein distance in \cite{RPDB2012} which turned out to be very useful in applications \cite{DLPYL2023,KNSBR2019}.
Kolouri et al.~\cite{kolouri22} adapted this idea for maximum mean discrepancies (MMD) and observed that the sliced MMD is again a MMD with a sliced kernel of the form \eqref{eq:sliced_kernel_intro}.
In \cite{HWAH2023}, the authors considered so-called Riesz kernels $K(x,y)=-\|x-y\|^r$, $r\in(0,2)$, which lead to the energy distance when they are inserted in the MMD, see e.g., \cite{SBGF2013}. They proved that sliced Riesz kernels are again Riesz kernels which offers the way to large scale applications, see \cite{HHABCS2023}. 
However, the theoretical analysis of Riesz kernels is challenging, see \cite{AHS2023,HGBS2024}.
Therefore the question arises if similar properties can be established for other kernels. In this paper, we want to address this question.
Some statements from this paper were generalized in the recent preprint \cite{RQS2024}.

\paragraph{Fast Kernel Summations}

The improvement of the quadratic complexity within kernel methods has a rich history in literature.
In a low dimensional setting, fast kernel summations are based on (non-)equispaced (fast) Fourier transforms \cite{GRB2022,KPS2006,PSN2004}, fast multipole methods \cite{BN1992,GR1987,LG2008,YR1999}, tree-based methods \cite{MXB2015,MXCB2015} or others \cite{H1999,MDHY2017, LP2022}.
For the Gaussian kernel, the fast Gauss transform was proposed in \cite{GS1991} and improved in \cite{YDGD2003,YDD2004}.
More general fast kernel transforms were proposed in \cite{RAGD2022}.
However, all of these approaches are limited to the low dimensional setting
and a generalization to the high-dimensional case does not appear to be straight-forward.
For higher dimensions the authors of \cite{MB2017} propose a decomposition and low-rank approach.
In the case of positive definite kernels, random Fourier features \cite{RR2007} are an alternative way for fast kernel summation in high dimensions. We include an explanation and a numerical comparison between the proposed slicing method and random Fourier features in Section~\ref{sec:cmp_RFF}.
Finally, the brute-force computation of the kernel sum can massively speeded up by the careful handling of several implementation aspects and GPU usage. For instance the KeOps package \cite{CFGCD2021} achieves state-of-the-art performance for a small to moderate number of data points. We include a run time comparison in Section~\ref{sec:cmp_keops}.

\section{Sliced Kernels}\label{sec:sliced_kernels}

A kernel $K\colon\R^d\times\R^d\to\R$ is called symmetric if $K(x,y)=K(y,x)$ for all $x,y\in\R^d$ and positive definite if for all $n\in\N$, $x_i\in\R^d$  and $a_i\in\R$ for $i=1,...,n$ it holds that
\begin{equation*}
\sum_{i=1}^n\sum_{j=1}^n a_i a_j K(x_i,x_j)\geq0
\end{equation*}
with equality if and only if $a_i=0$ for all $i=1,...,n$.
For $\xi\in\Sp^{d-1}$, we denote by $P_\xi\colon\R^d\to\R$ with $P_\xi(x)=\langle \xi,x\rangle$ the projection onto the one-dimensional subspace spanned by the direction $\xi$. 
In order to reduce the dimension, we consider sliced kernels of the form 
\begin{equation}\label{eq:sliced_kernel}
K(x,y)=\E_{\xi\sim \mathcal U_{\Sp^{d-1}}}[{\mathrm k}(P_\xi(x),P_\xi(y))],
\end{equation}
where ${\mathrm k}\colon\R\times\R\to\R$ is a one-dimensional kernel and $\mathcal U_{\Sp^{d-1}}$ is the uniform distribution on $\Sp^{d-1}$.
In this section, we examine the relation between sliced kernels and general kernels. 
First, we show that $K$ is positive definite whenever $\mathrm{k}$ is positive definite.
For radial kernels, the reverse statement is true as well, see \cite[Cor 4.11]{RQS2024}.
 
\begin{lemma}\label{slicedpositive}
Let ${\mathrm k}\colon\R\times\R\to\R$ be a positive definite kernel. Then, $K\colon\R^d\times\R^d\to\R$ defined by \eqref{eq:sliced_kernel} is positive definite.
\end{lemma}
\begin{proof}
Let $x_i \in \R^d$ and $a_i\in\R$ for all $i \in \{1,...,n\}$ with $n \in \mathbb N$. 
Since ${\mathrm k}$ is positive definite, we obtain that
\begin{align*}
\sum_{i,j=1}^n a_i a_j K(x_i,x_j) = \int_{\Sp^{d-1}} \underbrace{\sum_{i,j=1}^n a_i a_j {\mathrm k}(P_{\xi}(x_i),P_{\xi}(x_j))}_{\geq0} \dx \mathcal U_{\Sp^{d-1}} (\xi)\ge 0.
\end{align*}
Moreover, equality yields that the integrand is zero $\mathcal U_{\Sp^{d-1}}$-almost everywhere which implies by the positive definiteness of $\mathrm{k}$ that $a_i=0$ for all $i=1,...,n$.
\end{proof}

\subsection{Sliced Radial Kernels}

An important class of kernels are radial kernels, which have the form
$
K(x,y)=F(\|x-y\|)
$
for some basis function $F\colon\R_{\geq0}\to\R$. 
By definition radial kernels are symmetric.
Next, we represent the basis functions of sliced radial kernels $K$ in terms of the basis function $f$ of the one-dimensional counterpart $\mathrm{k}\colon\R\times\R\to\R$ with $\mathrm{k}( x,y) = f(|x-y|)$.
To this end, we use the notation $L^\infty_\mathrm{loc}\coloneqq\{f\colon\R_{\geq0}\to\R:f|_{[0,s]}\in L^\infty([0,s])\text{ for all }s>0\}$. Then, by the following proposition, the basis function of $K$ can be expressed by the \emph{generalized Riemann-Liouville fractional integral} transform defined by
\begin{equation}\label{eq:basis_function}
\mathcal S_d\colon L^\infty_\mathrm{loc}\to L^\infty_\mathrm{loc},\qquad \mathcal S_d(f)=F\quad\text{ with }\quad F(s)=\frac{2\Gamma(\frac{d}{2})}{\sqrt{\pi}\Gamma(\frac{d-1}{2})}\int_0^1f(ts)(1-t^2)^{\frac{d-3}{2}}\d t
\end{equation}
which is a special case of Erdelyi-Kober integrals, see \cite{L1971,SKM1993}. 
It was used in a closely related setting for computing the Radon transform of radial functions, see \cite[Prop 2.2]{R2003}, and is known to be injective, see \cite[(8)]{L1971}.
The proof of the proposition is given in Appendix~\ref{proof:sliced_kernels}.

\begin{proposition}\label{prop:sliced_kernels}
Let $\mathrm{k}\colon\R\times\R\to\R$ be a one-dimensional radial kernel with basis function $f\in L^\infty_\mathrm{loc}$.
Then, for $d\geq 2$ the following holds true.
\begin{enumerate}
\item[(i)] The kernel $K\colon\R^d\times\R^d\to\R$ from \eqref{eq:sliced_kernel} is given by $K(x,y)=F(\|x-y\|)$, where $F=\mathcal S_d(f)$.
\item[(ii)] For $d\geq 4$ it holds that $\mathcal S_d(f)$ is differentiable on $(0,\infty)$ for all $f\in L^\infty_\mathrm{loc}$. In particular, $\mathcal S_d$ is not surjective.
\end{enumerate}
\end{proposition}
Note that the assumption $f\in L^\infty_\mathrm{loc}$ can be lowered to $f|_{[0,C]}\in L^1([0,C])$ for $d\geq 3$ since then $(1-t^2)^{\frac{d-3}{2}}$ is in $L^\infty([0,1])$.
Part (ii) of the proposition implies that there exist radial kernels $K$ which cannot be written in the form \eqref{eq:sliced_kernel}.
Therefore, we additionally assume that the corresponding basis function $F$ admits a globally convergent Taylor expansion around zero.
Then, we construct in the next theorem a one-dimensional kernel $\mathrm{k}$ such that \eqref{eq:sliced_kernel} holds true. 
The proof is given in Appendix~\ref{proof:kernels_to_sliced_analytic}.

\begin{theorem}\label{thm:kernels_to_sliced_analytic}
Let $K\colon\R^d\times\R^d\to\R$ be a radial kernel with basis function $F\colon\R_{\geq0}\to\R$. 
Further assume that $F$ is an analytic function with globally convergent Taylor series in $0$ given by
\begin{equation*}
F(x)=\sum_{n=0}^\infty a_n x^n.
\end{equation*}
Then, $K$ is given by \eqref{eq:sliced_kernel} for the radial kernel $k\colon\R\times\R\to\R$ with basis function $f\colon\R_{\geq0}\to\R$ given by
\begin{equation*}
f(x)=\sum_{n=0}^\infty b_n x^n,\quad\text{with}\quad b_n=\frac{\sqrt{\pi}\Gamma(\frac{n+d}{2})}{\Gamma(\frac{d}{2})\Gamma(\frac{n+1}{2})} a_n.
\end{equation*}
\end{theorem}

\begin{figure}
\begin{table}[H]
\scalebox{.68}{
\begin{tabular}{c|c|c|c|c}
Kernel&$F(x)$&Power series of $F(x)$&$f(x)$&Power series of $f(x)$\\\hline
Gaussian &$\exp(-\frac{x^2}{2\sigma^2})$&$\sum_{n=0}^\infty \frac{(-1)^n}{2^n\sigma^{2n}n!} x^{2n}$&${_1}F_1(\tfrac{d}{2};\tfrac12;\tfrac{-x^2}{2\sigma^2})$&$\sum_{n=0}^\infty \frac{(-1)^n\sqrt{\pi}\Gamma(\frac{2n+d}{2})}{2^n\sigma^{2n} n!\Gamma(\frac{d}{2})\Gamma(\frac{2n+1}{2})}x^{2n}$\\[2ex]
Sliced Gaussian&${_1}F_1(\tfrac12;\tfrac{d}{2};\tfrac{-x^2}{2\sigma^2})$&$\sum_{n=0}^\infty \frac{(-1)^n\Gamma(\frac{d}{2})\Gamma(\frac{2n+1}{2})}{2^n\sigma^{2n} n!\sqrt{\pi}\Gamma(\frac{2n+d}{2})}x^{2n}$&$\exp(-\frac{x^2}{2\sigma^2})$&$\sum_{n=0}^\infty \frac{(-1)^n}{2^n\sigma^{2n}n!} x^{2n}$\\[2ex]
Laplacian &$\exp(-\alpha x)$&$\sum_{n=0}^\infty \frac{(-1)^n\alpha^n}{n!} x^{n}$&Appendix~\ref{app:laplace_hypergeom}&$\sum_{n=0}^\infty \frac{(-1)^n\alpha^n\sqrt{\pi}\Gamma(\frac{n+d}{2})}{n!\Gamma(\frac{d}{2})\Gamma(\frac{n+1}{2})} x^{n}$\\[2ex]
Sliced Laplacian&no closed form&$\sum_{n=0}^\infty \frac{(-1)^n\alpha^n\Gamma(\frac{d}{2})\Gamma(\frac{n+1}{2})}{n!\sqrt{\pi}\Gamma(\frac{n+d}{2})} x^{n}$&$\exp(-\alpha x)$&$\sum_{n=0}^\infty \frac{(-1)^n\alpha^n}{n!} x^{n}$\\[2ex]
Mat\'ern (Appendix~\ref{app:laplace_hypergeom})&$\frac{2^{1-\nu}}{\Gamma(\nu)}(\tfrac{\sqrt{2\nu}}{\beta}x)^\nu K_\nu(\tfrac{\sqrt{2\nu}}{\beta}x)$&only for $\nu\in\frac12\Z$&Appendix~\ref{app:laplace_hypergeom}&only for $\nu\in\frac12\Z$\\[2ex]
Negative Distance&$-x$&$-x$&$-\frac{\sqrt{\pi}\Gamma(\frac{d+1}{2})}{\Gamma(\frac{d}{2})} x$&$-\frac{\sqrt{\pi}\Gamma(\frac{d+1}{2})}{\Gamma(\frac{d}{2})} x$\\[2ex]
Riesz for $r\in(0,2)$ \cite{HWAH2023}&$-x^r$&does not exist&$-\frac{\sqrt{\pi}\Gamma(\frac{d+r}{2})}{\Gamma(\frac{d}{2})\Gamma(\frac{r+1}{2})}x^r$&does not exist\\
Thin Plate Spline&$x^2\log(x)$&does not exist&\begin{tabular}{@{}c@{}}$d x^2\log(x)-C x^2$\\ with $C$ from \eqref{eq:thin_pline_C}\end{tabular}&does not exist
\end{tabular}}
\caption{
Basis functions $F$ for different kernels $K(x,y)=F(\|x-y\|)$ and corresponding basis functions $f$ from  $\mathrm{k}(x,y)=f(|x-y|)$. 
}
\label{tab:kernels}
\end{table}

\begin{figure}[H]
\centering
\begin{subfigure}{.3\textwidth}
\includegraphics[width=\textwidth]{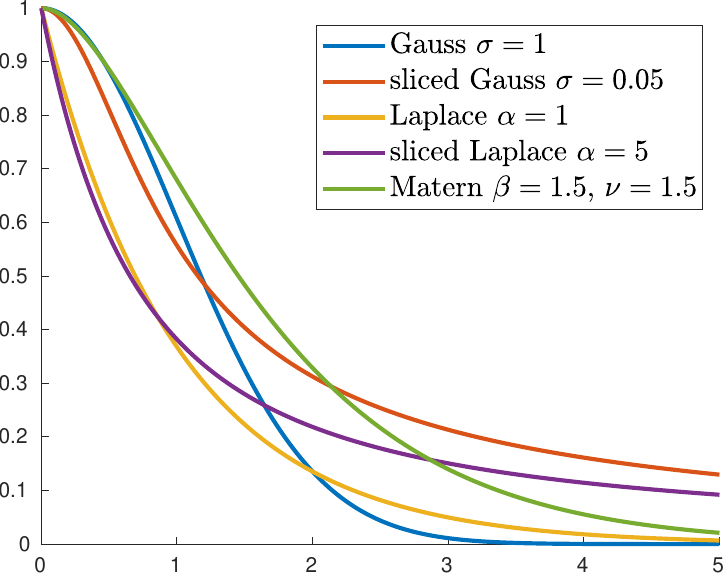}
\caption*{Basis functions $F$ for $d=10$}
\end{subfigure}
\qquad
\begin{subfigure}{.3\textwidth}
\includegraphics[width=\textwidth]{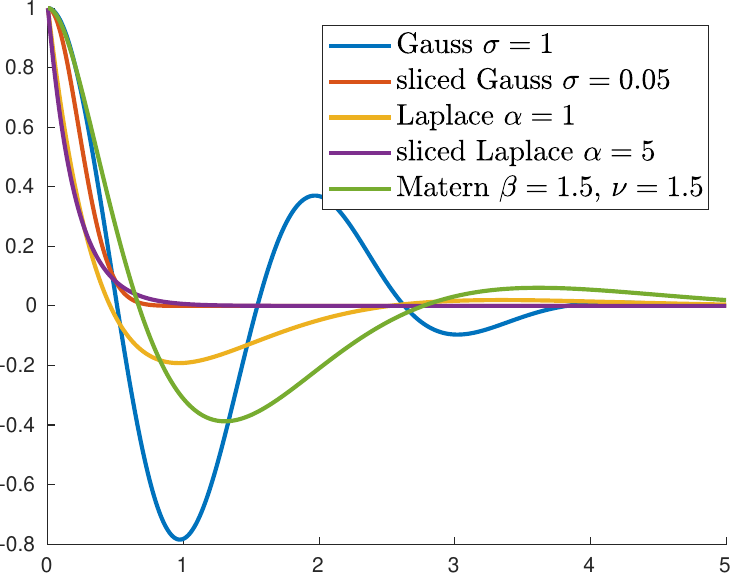}
\caption*{Basis functions $f$ for $d=10$}
\end{subfigure}

\begin{subfigure}{.3\textwidth}
\includegraphics[width=\textwidth]{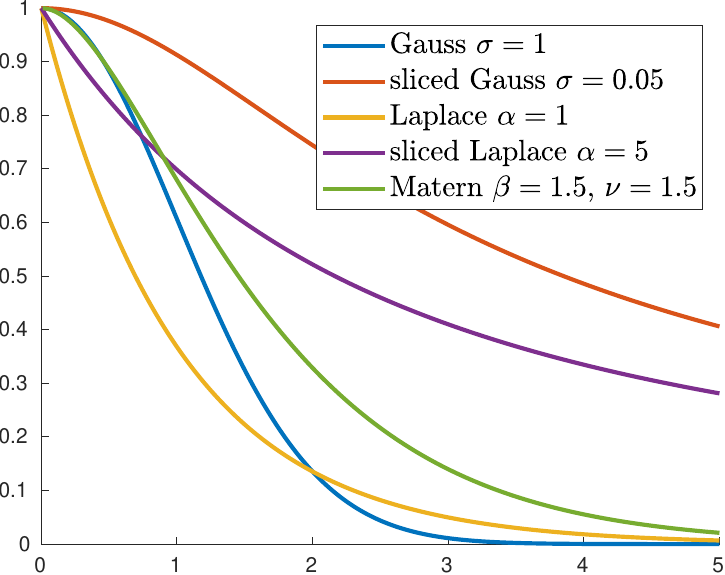}
\caption*{Basis functions $F$ for $d=100$}
\end{subfigure}
\qquad
\begin{subfigure}{.3\textwidth}
\includegraphics[width=\textwidth]{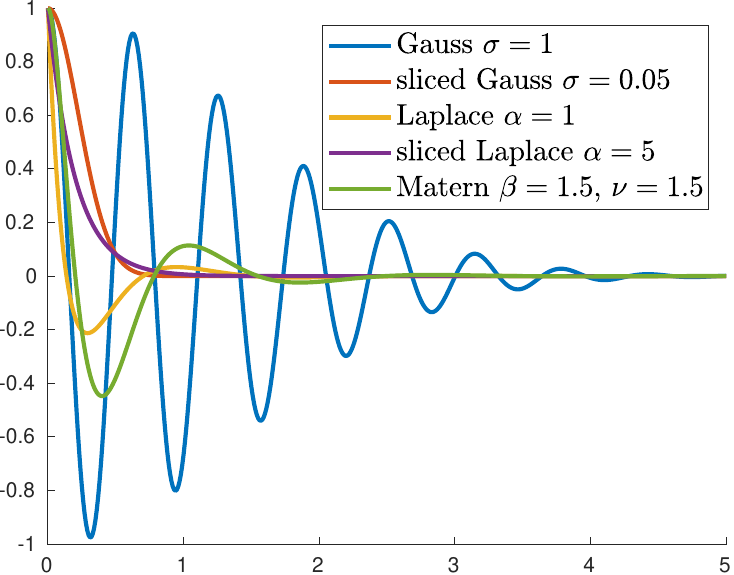}
\caption*{Basis functions $f$ for $d=100$}
\end{subfigure}
\caption{We plot the basis functions $F$ of the kernels $K(x,y)=F(\|x-y\|)$ from Table~\ref{tab:kernels} and the basis functions $f$ from the corresponding one-dimensional kernels $\mathrm{k}(x,y)=f(|x-y|)$.}
\label{fig:basis_functions}
\end{figure}
\end{figure}
\begin{example}
We apply the theorem to several kernels $K(x,y)=F(\|x-y\|)$ and compute the one-dimensional counterparts $\mathrm{k}(x,y)=f(|x-y|)$ in Table~\ref{tab:kernels}. In particular, we obtain the following.
\begin{itemize}
\item[-] For the Gaussian kernel, the function $f$ is given by the confluent hypergeometric function ${_1}F_1(a;b;x)$. This case was already considered in \cite{KSTP2020}.
\item[-] For Laplacian kernel and Mat\'ern kernel, $f$ is given by the generalized hypergeometric function ${_1}F_2(a;b,c;x)$. The Mat\'ern kernel is defined via the modified Bessel function of second kind $K_\nu$. Its series expansion is only for $\nu\in\frac12\Z$ really a power series. Nevertheless we can apply an analogous version of Theorem~\ref{thm:kernels_to_sliced_analytic} for general $\nu$. We include the details and the exact formula in Appendix~\ref{app:laplace_hypergeom}. Note that the Laplacian kernel coincides with the Mat\'ern kernel for $\nu=\frac12$ and $\beta=\frac{1}{\alpha}$.
\item[-] We could also set the one-dimensional kernel $\mathrm{k}$ to a Gaussian or Laplacian kernel and compute the corresponding high-dimensional kernel $K$. We include these examples under the name ``sliced Gaussian'' and ``sliced Laplacian'' kernel. For the Gaussian kernel, $F$ is again a confluent hypergeometric function.
\item[-] We include the derivation for the thin plate spline kernel and the corresponding constant $C$ in Appendix~\ref{app:laplace_hypergeom}.
\end{itemize}
We plot the arising basis functions $F$ and $f$ in Figure~\ref{fig:basis_functions}. It is remarkable that for the Gaussian, Laplacian and Mat\'ern kernel, the basis function $f$ is not monotone.
\end{example}

For radial kernels with non-analytic basis functions, we can still give an existence result of an one-dimensional counterpart for $d=3$.

\begin{corollary}\label{cor:existence_d3}
Let $K\colon\R^d\times\R^d\to\R$ be a radial kernel with differentiable basis function $F\colon\R_{\geq0}\to\R$. 
If $d=3$, then $\mathrm{k}(x,y)=f(|x-y|)$ with $f(t)=F(t)+tF'(t)$ fulfills \eqref{eq:sliced_kernel}.
\end{corollary}
\begin{proof}
We have to prove that \eqref{eq:basis_function} holds true for $F$ and $f$.
Integrating the definition
$
f(t)=F(t)+tF'(t)
$
from $0$ to $s$ gives
$
\int_0^sf(t)\d t=sF(s).
$
Dividing by $s$ and substituting $t$ by $t/s$ in the integral, we obtain
$
F(s)=\frac{1}{s}\int_0^sf(t)\d t=\int_0^1f(st)\d t
$.
For $d=3$, this results in \eqref{eq:basis_function}.
\end{proof}

It might be possible to generalize this approach: For odd $d$, we can reformulate \eqref{eq:basis_function} by integration by parts as an ODE. Unfortunately the ODE is nonlinear and the coefficient functions admit several singularities. Therefore the existence of solutions is highly unclear and a closer examination of these ODEs is not within the scope of this paper.

\subsection{The Gaussian Kernel}\label{sec:Gauss}

\begin{figure}
\centering
\begin{subfigure}{.24\textwidth}
\includegraphics[width=\textwidth]{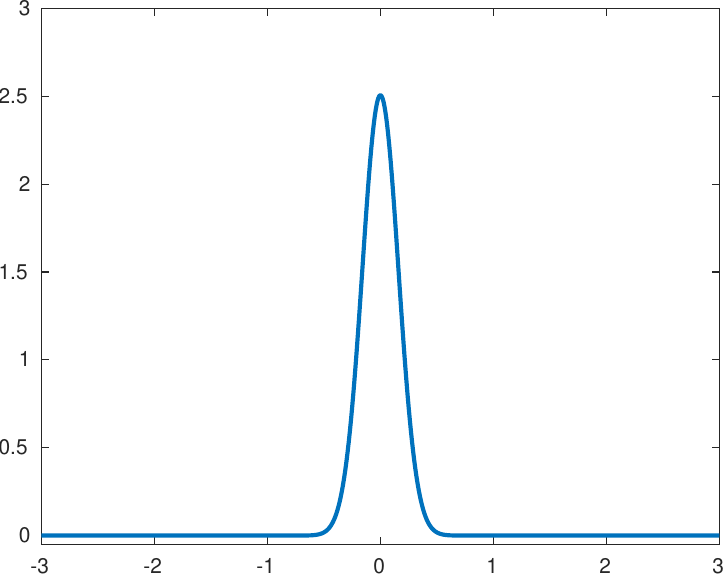}
\caption*{$d=1$}
\end{subfigure}
\begin{subfigure}{.24\textwidth}
\includegraphics[width=\textwidth]{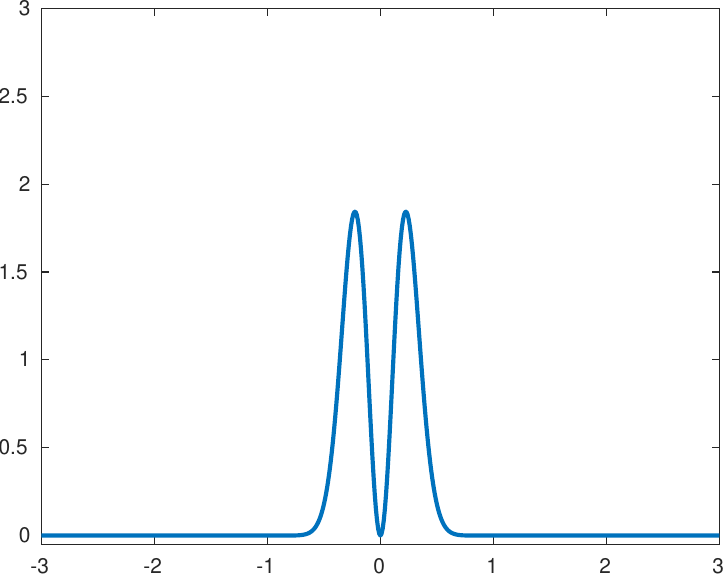}
\caption*{$d=3$}
\end{subfigure}
\begin{subfigure}{.24\textwidth}
\includegraphics[width=\textwidth]{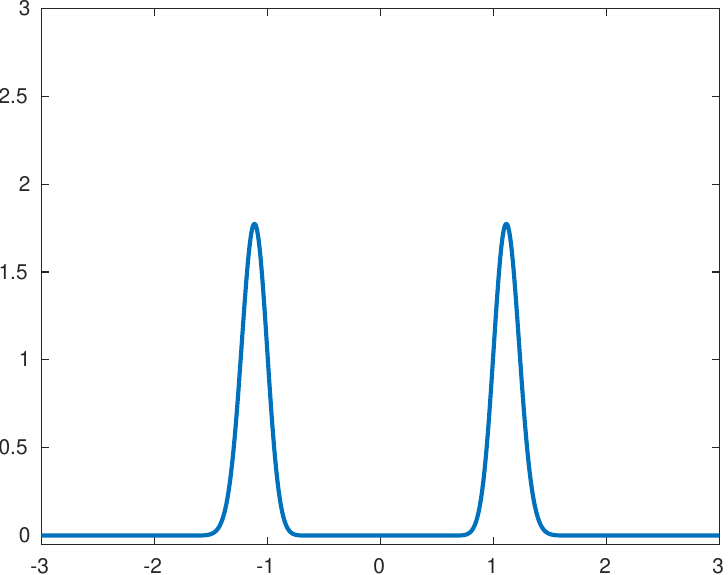}
\caption*{$d=50$}
\end{subfigure}
\begin{subfigure}{.24\textwidth}
\includegraphics[width=\textwidth]{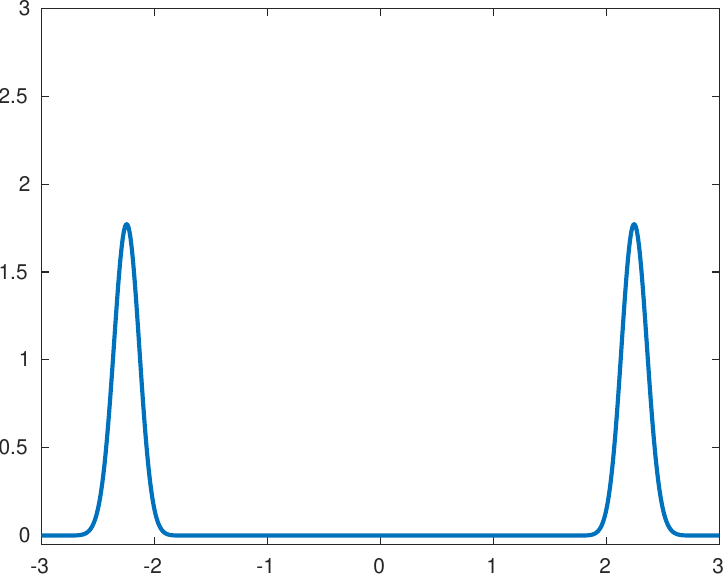}
\caption*{$d=200$}
\end{subfigure}
\caption{Plot of the Fourier transform $\hat f_1$ of $f_1(x)={_1}F_1(\tfrac{d}2,\tfrac12,\tfrac{-x^2}{2})$ for different dimensions $d$.}
\label{fig:plot_fourier_transform}
\end{figure}

We have seen in Table~\ref{tab:kernels} and Figure~\ref{fig:basis_functions} that the one-dimensional kernel $\mathrm{k}(x,y)=f_\sigma(|x-y|)$ corresponding to the $d$-dimensional Gaussian kernel is given by the confluent hypergeometric function $f_\sigma(x)={_1}F_1(\tfrac{d}{2};\tfrac12;\tfrac{-x^2}{2\sigma^2})$. Unfortunately, this is an highly oscillating function and therefore computational costly and numerical instable in its evaluation.
As a remedy, we compute its Fourier transform by the following lemma. The proof is a short calculation based on \cite[(13.10.12)]{OLBC2010}.

\begin{lemma}\label{lem:Fourier_Gauss}
For $\sigma>0$, the Fourier transform of $f_\sigma(x)={_1}F_1(\tfrac{d}{2};\tfrac12;\tfrac{-x^2}{2\sigma^2})$ is given by
\begin{equation*}
\hat f_\sigma(\omega)=\int_\R f_\sigma(x)\exp(-2\pi i x\omega)\d x=\frac{d\pi\sigma\exp(-2\pi^2\sigma^2\omega^2)(2\pi^2\sigma^2\omega^2)^{(d-1)/2}}{\sqrt{2}\Gamma(\frac{d+2}{2})}.
\end{equation*}
\end{lemma}
\begin{proof}
By \cite[(13.10.12)]{OLBC2010}, we have 
\begin{equation*}
\frac{1}{\sqrt{2\pi}}\int_\R f_1(x)\exp(-ix\omega)\d x=\frac{\sqrt{\pi}2^{-(d+1)/2}d\exp(-\omega^2/2)|\omega|^{d-1}}{\Gamma(\frac{d+2}{2})}.
\end{equation*}
Noting that we use a different notation of the Fourier transform than \cite{OLBC2010} (rescaling by $2\pi$) and using $f_\sigma(x)=f_1(x/\sigma)$ we obtain that
\begin{align}
\hat f_\sigma(\omega)&=\int_\R f_\sigma(x)\exp(-2\pi ix\omega)\d x
=\int_\R f_1(x/\sigma)\exp(-2\pi ix\omega)\d x\\
&=\sigma \int_\R f_1(x)\exp(- ix (2\pi\sigma\omega))\d x
=\frac{\sigma\pi2^{-d/2}d\exp(-(2\pi\sigma\omega)^2/2)|2\pi\sigma\omega|^{d-1}}{\Gamma(\frac{d+2}{2})}\\
&=\frac{d\pi\sigma\exp(-2\pi^2\sigma^2\omega^2)(2\pi^2\sigma^2\omega^2)^{(d-1)/2}}{\sqrt{2}\Gamma(\frac{d+2}{2})}
\end{align}
This concludes the proof.
\end{proof}

We plot $\hat f_\sigma$ for $\sigma=1$ and different choices of $d$ in Figure~\ref{fig:plot_fourier_transform}.
We can see that $\hat f_\sigma$ is almost zero despite on two peaks symmetrically located around $0$, where the distance between the peaks becomes larger for larger $d$.
Since $\hat f_\sigma$ converges exponentially fast to zero, we obtain for $\sigma$ small enough that
\begin{equation*}
f_\sigma(|x|)\approx \sum_{k\in Z}f_\sigma(|x+k|)\eqqcolon \phi(x),\quad x\in[-\tfrac12,\tfrac12).
\end{equation*}
By Poissons' summation formula (see, e.g., Theorem 2.26 in the text book \cite{PPST2018}) we have that
\begin{equation*}
\phi(x)=\sum_{k\in\Z}c_k\exp(-2\pi i k x)
\end{equation*}
with Fourier coefficients $c_k$ are given by
\begin{equation*}
c_k=\hat f_\sigma(k)=\frac{d\pi\sigma\exp(-2\pi^2\sigma^2k^2)(2\pi^2\sigma^2k^2)^{(d-1)/2}}{\sqrt{2}\Gamma(\frac{d+2}{2})}.
\end{equation*}
Note that the $c_k$ can be computed numerically stable in the log-space.
Moreover, we note that $c_k$ is only for very few $k$ significantly different from zero, see Figure~\ref{fig:plot_fourier_transform}. 
Therefore, we can truncate the Fourier transform by
\begin{equation*}
f_\sigma(|x|)\approx\phi(x)\approx\sum_{k\in\mathcal C}c_k\exp(-2\pi i k x)
\end{equation*}
where
\begin{equation}\label{eq:Fourier_cutoff}
\mathcal C=\{k\in\Z:-K_{\max}\le k\le K_{\max}, |c_k|>\epsilon\}
\end{equation}
for some small $\epsilon>0$ and $K_{\max}\in\Z_{>0}$. The thresholds $\epsilon$ and $K_{\max}$ are hyperparameters which have to be chosen properly. Due to the symmetry of $\hat f_\sigma$, we obtain that $k\in\mathcal C$ if and only if $-k\in \mathcal C$ such that the above approximation remains real-valued.
Considering $\hat f_\sigma$ from Figure~\ref{fig:plot_fourier_transform}, we see that we can approximate $f_\sigma$ with a small number $|\mathcal C|$ of Fourier coefficients, even though the elements of $\mathcal C$ are not close to zero for large $d$.

\subsection{Sliced Kernel Approximation and Slicing Error}\label{sec:slicing_err}

In practice, we evaluate sliced kernels of the form \eqref{eq:sliced_kernel} by replacing the expectation by a finite sum. That is, we approximate 
\begin{equation*}
K(x,y)\approx\frac1P\sum_{p=1}^P{\mathrm k}(P_{\xi_p}(x),P_{\xi_p}(y))
\end{equation*}
for random directions $\xi_1,...,\xi_P$.
The following theorem is based on Hoeffdings' inequality and bounds the error which is introduced by this approximation.

\begin{proposition}\label{thm:accuracy_P}
Let $x,y\in\R^d$ and let $K$ and $\mathrm{k}$ be related by \eqref{eq:sliced_kernel}. Moreover, assume that there exists $c_d(x,y)>0$ such that $|\mathrm{k}(z_1,z_2)|<c_d(x,y)$ for all $|z_1|,|z_2|<\max(\|x\|,\|y\|)$. Then, for $P$ iid directions $\xi_1,...,\xi_P\sim\mathcal U_{\mathbb S^{d-1}}$, it holds
\begin{equation*}
\mathbb P\Big[|\frac1P\sum_{p=1}^P\mathrm{k}(P_{\xi_p}(x),P_{\xi_p}(y))-K(x,y)|>t\Big]\leq 2\exp\Big(-\frac{P t^2}{2c_d(x,y)^2}\Big),
\end{equation*}
and
\begin{equation*}
\E\Big[\Big|\frac1P\sum_{p=1}^P\mathrm{k}(P_{\xi_p}(x),P_{\xi_p}(y))-K(x,y)\Big|\Big]\leq \frac{\sqrt{2\pi}c_d(x,y)}{\sqrt{P}} \in O\Big(\frac{c_d(x,y)}{\sqrt{P}}\Big).
\end{equation*}
\end{proposition}
\begin{proof}
\textbf{First inequality:}
Let $X_p=\frac1P\mathrm{k}(P_{\xi_p}(x),P_{\xi_p}(y))$ and $S_P=\sum_{p=1}^P X_p$. Since $|P_{\xi_p}(x)|=|\langle \xi_p,x\rangle|\leq \|\xi_p\|\|x\|=\|x\|$, we have by assumption that $-\frac{c_d(x,y)}{P}\leq X_p\leq \frac{c_d(x,y)}{P}$. Moreover, \eqref{eq:sliced_kernel} implies that $\E[S_P]=K(x,y)$.
Consequently, Hoeffding's inequality \cite{H1963} yields for all $t>0$ that
\begin{equation*}
\mathbb P\Big[|\frac1P\sum_{p=1}^P\mathrm{k}(P_{\xi_p}(x),P_{\xi_p}(y))-K(x,y)|>t\Big]=\mathbb P\Big[|S_P-\E[S_P]|>t\Big]\leq 2\exp\Big(-\frac{P t^2}{2c_d(x,y)^2}\Big).
\end{equation*}
\textbf{Second inequality:}
Denote by $X$ the random variable
\begin{align*}
X=|\frac1P\sum_{p=1}^P\mathrm{k}(P_{\xi_p}(x),P_{\xi_p}(y))-K(x,y)|.
\end{align*}
Then, we have by the first part that
\begin{align*}
\mathbb P[X>t]\leq 2\exp\Big(-\frac{P t^2}{2c_d(x,y)^2}\Big).
\end{align*}
Thus, we obtain
\begin{align*}
\E[X]=\int_0^\infty \mathbb P[X>t]\d t\leq 2\int_0^\infty \exp{\left(-\frac{P\ t^2}{2\ c_d(x,y)^2}\right)}\d t=\frac{\sqrt{2\pi}c_d(x,y)}{\sqrt{P}},
\end{align*}
where the last step follows from the identity $\int_0^\infty \exp(-t^2)\d t=\frac{\sqrt{\pi}}{2}$.
Inserting the definition of $X$, this completes the proof.
\end{proof}

The proposition yields a dimension-independent error bound whenever $\mathrm{k}$ is bounded independently from the dimension.
By the following corollary, this is in particular the case for the Gaussian kernel.
In contrast, for the negative distance kernel, the one-dimensional counterpart $\mathrm{k}$ is not bounded and its scaling depends on the dimension. Thus, the resulting error rate depends both on the dimension and the diameter of the data points.

\begin{theorem}[Error Bounds]\label{cor:error_bound}
Let $K\colon\R^d\to\R^d$ and $\mathrm{k}\colon\R\to\R$ be related by \eqref{eq:sliced_kernel}.
\begin{enumerate}
\item[(i)] Assume that $\mathrm{k}(x,y)\leq C$ for some $C$ independent of $d$. Then, it holds
\begin{equation*}
\E\Big[\Big|\frac1P\sum_{p=1}^P\mathrm{k}(P_{\xi_p}(x),P_{\xi_p}(y))-K(x,y)\Big|\Big]\leq  \frac{\sqrt{2\pi}C}{\sqrt{P}}\in O\Big(\frac{1}{\sqrt{P}}\Big).
\end{equation*}
Moreover, this assumption is fulfilled for the Gaussian kernel $K(x,y)=\exp(-\tfrac{\|x-y\|^2}{2\sigma^2})$ and its one-dimensional counterpart $\mathrm{k}(x,y)={_1}F_1(\tfrac{d}{2};\tfrac12;\tfrac{-|x-y|}{2\sigma^2})$ with $C=1$.
\item[(ii)] Let $K(x,y)=-\|x-y\|$ and $\mathrm{k}(x,y)=-\frac{\sqrt{\pi}\Gamma(\frac{d+1}{2})}{\Gamma(\frac{d}{2})}|x-y|$ its one-dimensional counterpart. Then it holds
\begin{equation*}
\E\Big[\Big|\frac1P\sum_{p=1}^P\mathrm{k}(P_{\xi_p}(x),P_{\xi_p}(y))-K(x,y)\Big|\Big]\leq  \frac{\sqrt{8}\pi \Gamma(\frac{d+1}{2})\|x-y\|}{\Gamma(\frac{d}{2})\sqrt{P}}\in O\Big(\frac{\sqrt{d}\|x-y\|}{\sqrt{P}}\Big).
\end{equation*}
\end{enumerate}
\end{theorem}
\begin{proof}
\begin{enumerate}
\item[(i)] The first statement follows directly by Proposition~\ref{thm:accuracy_P}.
Thus, it remains to show that $|f_\sigma(x)|\leq 1$.
Using Hölders inequality, the Fourier inversion formula and the formula for $\hat f_\sigma$ from Lemma~\ref{lem:Fourier_Gauss}, we have
\begin{align*}
|f_\sigma(x)|
&=
\Big|\int_\R \hat f_\sigma(\omega)\exp(2\pi i \omega x)\d \omega\Big|
\leq
\int_\R|\hat f_\sigma(\omega)|\d \omega\|\exp(2\pi i x \cdot)\|_{L^\infty}=\int_\R\hat f_\sigma(\omega)\d \omega
\\
&=
\frac{d\pi\sigma}{\sqrt{2}\Gamma(\frac{d}{2}+1)}\int_\R\exp(-2\pi^2\sigma^2\omega^2)(2\pi^2
\sigma^2\omega^2)^{(d-1)/2}\d \omega
\\
&=\frac{1}{\Gamma(\frac{d}{2})}\int_\R\exp(-\omega^2)(\omega^2)^{(d-1)/2}\d \omega,
\end{align*}
where we substituted $\sqrt{2}\pi\sigma\omega$ by $\omega$ and used $\frac{d}{2}\Gamma(\frac{d}{2})=\Gamma(\frac{d}{2}+1)$ in the last step. By the symmetry of the integral and substituting $\omega^2$ by $\omega$, this is equal to
\begin{equation*}
\frac{2}{\Gamma(\frac{d}{2})}\int_0^\infty\exp(-\omega^2)(\omega^2)^{(d-1)/2}\d \omega=\frac{1}{\Gamma(\frac{d}{2})}\int_0^\infty\exp(-\omega)\omega^{d/2-1}\d \omega=1,
\end{equation*}
where the last step is the definition of the Gamma function.

\item[(ii)] Note that it holds $K(x,y)=K(x-y,0)$ and $\mathrm{k}(P_{\xi}(x),P_{\xi}(y))=\mathrm{k}(P_{\xi}(x-y),0)$. Moreover, it holds $|\mathrm{k}(z_1,z_2)|=\frac{\sqrt{\pi}\Gamma(\frac{d+1}{2})}{\Gamma(\frac{d}{2})}|z_1-z_2|\leq \frac{2\sqrt{\pi}\Gamma(\frac{d+1}{2})}{\Gamma(\frac{d}{2})}\|x-y\|$ for all $|z_1|,|z_2|\leq\|x-y\|$. Inserting this in Proposition~\ref{thm:accuracy_P} and applying the bound $\frac{\Gamma(\frac{d+1}{2})}{\Gamma(\frac{d}{2})} \leq \sqrt{\frac{d}{2}+\sqrt{\frac{3}{4}}-1}$ proven in \cite{Kershaw83} yields the claim.
\end{enumerate}
\end{proof}

We conjecture that the assumptions of part (i) are also true for the Laplacian and Mat\'ern kernel for $C=1$, even though we are not able to provide a formal proof. The plots in Figure~\ref{fig:basis_functions} and the numerical examples in Section~\ref{sec:numerics} seem to confirm this conjecture. Moreover, the rate from part (ii) can be refined by a closer analysis to $O\big(\frac{\|x-y\|}{\sqrt{P}}\big)$, see \cite{HJQ2024}. Note that for many applications (and in the numerical examples in Section~\ref{sec:numerics}) the diameter $\|x-y\|$ still depends on the dimension $d$ by $O(\sqrt{d})$.

\section{Fast Kernel Summation via Slicing}\label{sec:fast_kernel_summation}

Next, we apply our findings from the previous section in order to compute 
kernel summations of the form
\begin{equation*}
s_m=\sum_{n=1}^Nw_nK(x_n,y_m)
\end{equation*}
in a fast way.
To this end, we assume that $K(x,y)=F(\|x-y\|)$ is a radial kernel of the form \eqref{eq:sliced_kernel} for some $\mathrm{k}(x,y)=f(|x-y|)$.
Then we have that
\begin{equation*}
s_m=\E_{\xi\sim \mathcal U_{\Sp^{d-1}}}\Big[\sum_{n=1}^Nw_n\mathrm{k}(P_\xi(x_n),P_\xi(y_m))\Big]\approx \frac1P\sum_{p=1}^P \sum_{n=1}^Nw_n\mathrm{k}(P_{\xi_p}(x_n),P_{\xi_p}(y_m)),
\end{equation*}
where the approximation error in the last step can be bounded by Proposition~\ref{thm:accuracy_P} and Theorem~\ref{cor:error_bound}.
In particular, we can compute $P$ one-dimensional kernel sums instead of one $d$-dimensional kernel sum.
In this section, we focus on the computation of the one-dimensional kernel summations
\begin{equation*}
t_m=\sum_{n=1}^Nw_n\mathrm{k}(x_n,y_m)
\end{equation*}
for $x_1,...,x_N,y_1,...,y_M\in\R$. 
Then, the arising computation procedure of $s_m$ is outlined in Algorithm~\ref{alg:fast_Fourier_sliced}.
Note that the method can also treat vector-valued kernels, see e.g.~\cite{CDT2006}, by computing the kernel sums for every component separately.

As outlined in the ``related work'' section in the introduction, there is long history in the literature for computing the $t_m$ computationally efficient.
If $\mathrm{k}$ is a smooth kernel, we apply here the fast Fourier summation from \cite{KPS2006,PSN2004}, see also the text book \cite[Sec~7.5]{PPST2018}.
We revisit the fast Fourier summation in Subsection~\ref{sec:1d}.
Moreover, we consider two important examples of non-smooth kernels.
For the negative distance kernel, we propose a sorting algorithm based on \cite{HWAH2023} for computing the $t_m$ in Subsection~\ref{sec:energy}.
Finally, we decompose the Laplacian kernel into a smooth part and a negative distance part and combine the two previous approaches in Subsection~\ref{sec:laplacian}.

\begin{algorithm}[t]
\begin{algorithmic}
\State Draw $P$ random directions $\xi_1,...,\xi_P$ from $\mathcal U_{\mathbb S^{d-1}}$.
\For{$p=1....,P$}
\State Compute $t_{j,p}=\sum_{i=1}^Nw_i\mathrm{k}(P_{\xi_p}(x_i),P_{\xi_p}(y_j))$ via one-dimensional fast kernel summation.
\EndFor
\State Average the results by $s_j=\frac1P\sum_{p=1}^Pt_{j,p}$
\end{algorithmic}
\caption{Sliced Fast Fourier Summation}
\label{alg:fast_Fourier_sliced}
\end{algorithm}

\subsection{Smooth Kernels: Fast Fourier Summation}\label{sec:1d}

Let $\phi\colon[-\tfrac12,\tfrac12)\to\R$ be a periodic square-integrable function.
Later, $\phi$ will be related to some one-dimensional radial kernel $\mathrm{k}(x,y)=f(|x-y|)$ by $\mathrm{k}(x,y)=\phi(x-y)$ for all $x,y\in[-\frac12,\frac12)$ with $|x-y|<\tfrac12$, that is $\phi=f(|\cdot|)$.
Moreover, let $x_1,...,x_N\in[-\frac12,\frac12)$ and $y_1,...,y_M\in[-\frac12,\frac12)$ such that $\max_{n=1,...,N,m=1,...,M}|x_n-y_m|<\tfrac12$ and let $w_1,...,w_N\in\R$ be some weights.
Then, we aim to compute the sums
\begin{equation}\label{eq:large_sum}
t_m\coloneqq\sum_{n=1}^N w_n \phi(x_n-y_m)\quad\text{for all }m=1,...,M
\end{equation}
in a fast way.
Naively, computing all sums in \eqref{eq:large_sum} requires to evaluate all $MN$ summands.
In order to improve this computational complexity, we apply fast Fourier summations, see \cite{PPST2018}. 
More precisely, we expand $\phi$ in its Fourier series and truncate it by
\begin{equation}\label{eq:Fourier_truncation}
\phi(x)=\sum_{k\in\Z}c_k\exp(-2\pi ikx)\approx\sum_{k\in \mathcal C}c_k\exp(-2\pi ikx),
\end{equation}
where $\mathcal C$ is a finite subset of $\Z$.
In order to ensure that the right side of \eqref{eq:Fourier_truncation} remains real-valued, we assume that $k\in\mathcal C$ if and only if $-k\in\mathcal C$.
Then, \eqref{eq:large_sum} reads as
\begin{equation*}
t_m=\sum_{n=1}^N\sum_{k\in \mathcal C} w_nc_k\exp(2\pi ik(y_m-x_n))=
\sum_{k\in\mathcal C}c_k\exp(2\pi iky_m)\underbrace{\sum_{i=n}^N w_n\exp(-2\pi ikx_n)}_{\eqqcolon \hat w_k}
\end{equation*}
Now, the computation of the coefficients $\hat w=(\hat w_k)_{k\mathcal C}$ is given by the matrix-vector mulitplication $\hat w=\mathcal F_{\mathcal C,x}^{\adj} w$ of the matrix $\mathcal F_{\mathcal C,x}^{\adj}=(\exp(-2\pi ikx_n))_{k\in C,n=1}^{N}$ with the vector $w=(w_n)_{n=1}^N$.
This corresponds to the adjoint discrete Fourier transform on the non-equispaced grid $x=(x_1,...,x_N)$ (NDFT).
Once $\hat w$ is computed, we can derive the sums $t=(t_m)_{m=1}^M$ by
\begin{equation*}
t=\sum_{k\in \mathcal C}c_k\hat w_k\exp(2\pi iky_m)=\mathcal F_{\mathcal C,y} (c\odot \hat w)=\mathcal F_{\mathcal C,y} (c\odot (\mathcal F_{\mathcal C,x}^{\adj} w)),
\end{equation*}
where $\mathcal F_{\mathcal C,y}=(\exp(-2\pi iky_m))_{m=1,k\in\mathcal C}^{M}$ is the discrete Fourier transform on the non-equisapced grid $y=(y_1,...,y_M)$ and $(c\odot \hat w)$ is the elementwise multiplication of the vectors $c=(c_k)_{k\in\mathcal C}$ and $\hat w$.

For the computation of $\mathcal F_{\mathcal C,x}^{\adj} w$, we have to evaluate $2|\mathcal C|N$ terms, for the elementwise product of $c$ and $\hat w$ we need $|\mathcal C|$ operations and for the multiplication with $\mathcal F_{\mathcal C,y}$ we require $2|\mathcal C|M$ operations.
Summarized this is a computational complexity of $O((M+N)|\mathcal C|)$, which is a significant improvement over $O(MN)$ as long as the number $|\mathcal C|$ of used Fourier coefficients is small.

\begin{remark}[Non-equispaced FFT]
For large $|\mathcal C|$ the computational complexity
$O((M+N)|\mathcal C|)$ can be improved to $O(M+N+|\mathcal C|\log(|\mathcal C|))$ by using the non-equispaced fast Fourier transforms (NFFT) \cite{B1995,DR1993,PST2001}. For a detailed explanation, we refer to the text book \cite{PPST2018} and the references therein.

In the numerical part, we use either the NDFT or the NFFT depending on the size of $|\mathcal C|$. For the Gaussian kernel, it turns out that $|\mathcal C|$ can be chosen very small, but the elements of $\mathcal C$ are not centered around zero, see Section~\ref{sec:Gauss}. Consequently, the NDFT is faster in this case.
For other kernels, like the Mat\'ern or the Laplacian kernel, we use the NFFT with $\mathcal C=\{-K,...,K\}$ for some $K\in \N$.
\end{remark}

\begin{remark}[Smoothness]\label{rem:smoothness}
The run time of the fast Fourier summation heavily depends on the number $|\mathcal C|$ of Fourier coefficients which is required to keep the approximation error in \eqref{eq:Fourier_truncation} small.
This number is again known to depend on the smoothness of the function $\phi$ by the Bernstein theorem, see, e.g., \cite[Thm 1.39]{PPST2018}. 
For non-smooth functions, the Fourier coefficients may decay very slowly, such that very large choices of $|\mathcal C|$ are required to keep the approximation error in \eqref{eq:Fourier_truncation} small. This leads to an intractable slow-down of the method.
As a remedy, one can include additional error correction methods for excluding points close to the non-smooth part of $\phi$, see e.g., \cite{TSGSW2011}.
For some kernels like the negative distance and Laplacian kernel, we propose alternative algorithms in Subsection~\ref{sec:energy} and Subection~\ref{sec:laplacian}.
\end{remark}

\paragraph{Rescaling}
So far, we assumed that $x_1,...,x_N\in[-\frac12,\frac12)$ and $y_1,...,y_M\in[-\frac12,\frac12)$ such that $\max_{n=1,...,N,m=1,...,M}|x_n-y_m|<\tfrac12$. In practice, this assumption might be restrictive.
Therefore, we rescale the summations \eqref{eq:large_sum},
Let $c_{\min}=\min(x_1,...,x_N,y_1,...,y_M)$ and $c_{\max}=\max(x_1,...,x_N,y_1,...,y_M)$ and let $T<0.25$ be a threshhold.
Then, we have that
\begin{equation*}
t_m=\sum_{n=1}^N w_n \phi(x_n-y_m)=\sum_{n=1}^N w_n \tilde \phi(\tilde x_n-\tilde y_m)
\end{equation*}
with
\begin{equation*}
\tilde \phi(x)=\phi(\frac{x+T}{\tau}),\quad \tilde x_n=\tau x_n-T\in(-0.25,0.25) ,\quad \tau=\frac{2T}{c_{\max}-c_{\min}}.
\end{equation*}
In particular, we have that it holds $\tilde x_1,...,\tilde x_N\in[-\frac12,\frac12)$ and $\tilde y_1,...,\tilde y_M\in[-\frac12,\frac12)$ such that $\max_{n=1,...,N,m=1,...,M}|\tilde x_n-\tilde y_m|<\tfrac12$.

\begin{remark}
Note that the scaling of the data points corresponds to an inverse scaling in the frequency domain. Consequently, the rescaling procedure comes with the cost of requiring a larger number $|\mathcal C|$ of Fourier coefficients. However, it is necessary in order to make the discrete Fourier transform applicable.
\end{remark}

\subsection{Negative Distance Kernel: Sorting Algorithm}\label{sec:energy}

Let $\mathrm{k}(x,y)=-c_d|x-y|$ with $c_d=\frac{\sqrt{\pi}\Gamma(\frac{d+1}{2})}{\Gamma(\frac{d}2)}$ be the one-dimensional counterpart of $K(x,y)=-\|x-y\|$.
Note that this kernel is not smooth such that the fast Fourier summation from Section~\ref{sec:1d} is not applicable.
In \cite{HWAH2023}, the authors proposed an efficient sorting algorithm to compute the derivative of maximum mean discrepancies.
Even though the algorithm from \cite{HWAH2023} is not directly applicable to the fast summation method, we can use similar ideas for computing the coefficients $s_m$. The resulting sorting algorithm is exact and does not induce any error like the approach from the previous subsection.

In the following, let $(z_1,...,z_{N+M})=(y_1,...,y_M,x_1,...,x_N)$ and denote $(v_1,...,v_{N+M})=(0,...,0,w_1,...,w_N)$.
Then, it holds
\begin{equation*}
t_m=-\sum_{n=1}^Nc_dw_n|x_n-y_m|=-\sum_{n=1}^{N+M}c_dv_n|z_n-z_m|.
\end{equation*}
Thus, we extend the definition of $t$ to $m>M$ by $t_m=-\sum_{n=1}^{N+M}c_dv_n|z_n-z_m|$.

\paragraph{Case $z_1\leq...\leq z_{N+M}$}

First we add a zero, such that
\begin{equation*}
t_m=-\sum_{n=1}^Nc_dw_n|x_n-y_m|=-\sum_{n=1}^{N+M}c_dv_n|z_n-z_m|.
\end{equation*}
For $m>n$ it holds $|z_n-z_m|=\sum_{i=n+1}^m z_i-z_{i-1}$ and for $m<n$ we have $|z_n-z_m|=\sum_{i=m+1}^n z_i-z_{i-1}$. Thus it holds
\begin{equation*}
t_m=-\sum_{n=1}^{m-1}c_dv_n\sum_{i=n+1}^m z_i-z_{i-1}-\sum_{n=m+1}^{N+M}c_dv_n\sum_{i=m+1}^nz_i-z_{i-1}.
\end{equation*}
Counting the occurrences of $z_i-z_{i-1}$ for $i=2,...,N+M$, we obtain
\begin{equation*}
t_m=-\sum_{i=2}^{m}\Big(\sum_{n=1}^{i-1}c_dv_n\Big) (z_i-z_{i-1})-\sum_{i=m+1}^{N+M}\Big(\sum_{n=i}^{N+M}c_dv_n\Big) (z_i-z_{i-1})
\end{equation*}
Using 
\begin{equation*}
\sum_{n=i}^{N+M}c_dv_n=\sum_{n=1}^{N+M}c_dv_n-\sum_{n=1}^{i-1}c_dv_n
\end{equation*}
this becomes
{\small
\begin{align*}
t_m&=-\sum_{i=2}^{m}\Big(\sum_{n=1}^{i-1}c_dv_n\Big) (z_i-z_{i-1})+\sum_{i=m+1}^{N+M}\Big(\sum_{n=1}^{i-1}c_dv_n\Big) (z_i-z_{i-1})-\Big(\sum_{n=1}^{N+M}c_dv_n\Big)\sum_{i=m+1}^{N+M} (z_i-z_{i-1})\\
&=-\sum_{i=2}^{m}\Big(\sum_{n=1}^{i-1}c_dv_n\Big) (z_i-z_{i-1})+\sum_{i=m+1}^{N+M}\Big(\sum_{n=1}^{i-1}c_dv_n\Big) (z_i-z_{i-1})-\Big(\sum_{n=1}^{N+M}c_dv_n\Big) (z_{N+M}-z_{m}).
\end{align*}}%
With 
\begin{equation*}
\sum_{i=m+1}^{N+M}\Big(\sum_{n=1}^{i-1}c_dv_n\Big) (z_i-z_{i-1})=\sum_{i=2}^{N+M}\Big(\sum_{n=1}^{i-1}c_dv_n\Big) (z_i-z_{i-1})-\sum_{i=2}^{m}\Big(\sum_{n=1}^{i-1}c_dv_n\Big) (z_i-z_{i-1})
\end{equation*}
we obtain
\begin{equation*}
t_m=\sum_{i=2}^{N+M}\Big(\sum_{n=1}^{i-1}c_dv_n\Big) (z_i-z_{i-1})-2\sum_{i=2}^{m}\Big(\sum_{n=1}^{i-1}c_dv_n\Big) (z_i-z_{i-1})-\Big(\sum_{n=1}^{N+M}c_dv_n\Big) (z_{N+M}-z_{m}).
\end{equation*}
Now, using the notations $a_{i-1}=\sum_{n=1}^{i-1}c_dv_n$, $b_{m-1}=\sum_{i=2}^{m}a_{i-1} (z_i-z_{i-1})=\sum_{i=1}^{m-1}a_i(z_{i+1}-z_i)$, we obtain that
\begin{equation*}
t_m=b_{M+N-1}-2b_{m-1}-a_{N+M}(z_{N+M}-z_m)
\end{equation*}
the following linear-time algorithm to compute all $t_m$ for $m=1,...,N+M$ as outlined in Algorithm~\ref{alg:fastsum_energy} using the cumsum function.

\begin{algorithm}[t]
\begin{algorithmic}
\State Input: $z_1\leq ...\leq z_{N+M}$, $v_1,...,v_{N+M}\in\R$.
\State Output: $t_m=-\sum_{n=1}^{N+M}c_dv_n|z_n-z_m|$.
\State Set $a=\mathrm{cumsum}(c_d v)$.
\State Set $\tilde a_i=a_i(z_{i+1}-z_i)$, for $i=1,...,N+M$.
\State Set $b=\mathrm{cumsum}(\tilde a)$.
\State Set $t_m=b_{M+N-1}-2b_{m-1}-a_{N+M}(z_{N+M}-z_m)$, for $m=1,...,N+M$.
\end{algorithmic}
\caption{Fast summation for the negative distance kernel}
\label{alg:fastsum_energy}
\end{algorithm}

\paragraph{General case}

Let $\sigma\colon\{1,...,N+M\}\to\{1,...,N+M\}$ be a permutation such that $z_{\sigma(1)}\leq ...\leq z_{\sigma(N+M)}$.
By interchanging the summation order, we obtain
\begin{equation*}
t_m=-\sum_{n=1}^{N+M}c_dv_{\sigma(n)}|z_{\sigma(n)}-z_{m}|
\end{equation*}
such that we obtain for $i=\sigma^{-1}(m)$ that
\begin{equation*}
t_{\sigma(i)}=-\sum_{n=1}^{N+M}c_dv_{\sigma(n)}|z_{\sigma(n)}-z_{\sigma(i)}|
\end{equation*}
Now, $t_{\sigma(i)}$ for $i=1,...,N+M$ can be computed as by the previous. Applying the inverse permutation $\sigma^{-1}$
gives the result.

\subsection{Laplacian Kernel: Decomposition}\label{sec:laplacian}

Finally, let $K(x,y)=\exp(-\alpha\|x-y\|)$ be the Laplacian kernel and $\mathrm{k}$ its one-dimensional counterpart. Since $K$ and $\mathrm{k}$ are not differentiable
we cannot apply the fast Fourier summation from Section~\ref{sec:fast_kernel_summation} directly.
Instead, we decompose $K$ as
\begin{equation*}
K=K_1+K_2,\quad K_1(x,y)=\exp(-\alpha\|x-y\|)+\alpha\|x-y\|,\quad K_2(x,y)=-\alpha\|x-y\|.
\end{equation*}
We denote the corresponding one-dimensional kernels by $\mathrm{k}_1$ and $\mathrm{k}_2$, which can be expressed as linear combinations of the kernels from Table~\ref{tab:kernels}.
Now, it is easy to check that $K_1$ and $\mathrm{k}_1$ are differentiable such that the fast summation from Section~\ref{sec:fast_kernel_summation} can be applied for computing
$
\sum_{n=1}^N w_n\mathrm{k}_1(x_n,y_m).
$
Moreover, $\mathrm{k}_2$ is the negative distance kernel such that we can also compute $\sum_{n=1}^N w_n\mathrm{k}_2(x_n,y_m)$ by the sorting algorithm outlined in the previous subsection.
Putting both summations together, we can compute
\begin{equation*}
t_m=\sum_{n=1}^N w_n\mathrm{k}(x_n,y_m)=\sum_{n=1}^N w_n\mathrm{k}_1(x_n,y_m)+\sum_{n=1}^N w_n\mathrm{k}_2(x_n,y_m).
\end{equation*}

\section{Numerical Examples}\label{sec:numerics}

Within this section, we aim to verify our findings numerically. It is not the purpose of these examples to demonstrate the power of kernel methods since this was considered extensively in the literature. Instead, we want to highlight the computational advantages of our method. To this end, we provide a run time 
comparison and error analysis for the fast kernel summation for different kernels. 

Afterwards, we compare slicing to random Fourier features. Finally, we discuss some scalability aspects and present a large-scale GPU comparison of slicing with KeOps \cite{CFGCD2021}.

The code for our experiments is written in PyTorch and Julia and is available online\footnote{available at \url{https://github.com/johertrich/sliced_kernel_fastsum}}. 
For computing the NFFT, we use the torch-NFFT package\footnote{\url{https://github.com/dominikbuenger/torch_nfft}} in PyTorch and the NFFT3-package \cite{KKP2009,S2018}\footnote{\url{https://www-user.tu-chemnitz.de/~potts/nfft/}} in Julia.

\subsection{Experimental Setup}

In order to compare the run time and to analyze the error of the fast kernel summation, we draw $N$ samples
$x_1,...,x_N\in\R^d$ and $M=N$ samples $y_1,...,y_M\in\R^d$ from $\mathcal N(0,0.1^2 I)$. We set $d=50$ and choose the weights $w_1,...,w_N$ randomly uniformly distributed in $[0,1]$.
Then, we compute the kernel summations 
\begin{equation*}
s_m=\sum_{n=1}^N w_n K(x_n,y_m),\quad m=1,...,M
\end{equation*}
for the Gaussian kernel with $\sigma=1$, the Mat\'ern kernel with $\nu=3/2$ and $\beta=1$, the negative distance kernel and the Laplacian kernel for $\alpha=0.5$.
In all cases, we first compute the kernel sums exactly and then via the fast kernel summation as proposed in Section~\ref{sec:fast_kernel_summation}.
For the Gaussian kernel we use the NDFT with the Fourier coefficients chosen as in \eqref{eq:Fourier_cutoff}.
We decompose the Laplacian kernel into a smooth and a negative distance kernel as outlined in Section~\ref{sec:laplacian}, where we use the NFFT with 200 Fourier coefficients for the smooth part and the sorting algorithm from Section~\ref{sec:energy} for the negative distance part.
For the Mat\'ern kernel, we use the NFFT with 1024 Fourier coefficients. The number of Fourier coefficients are manually chosen as small as possible such that the error in the non-equispace Fourier transform is not relevant. An automatic selection criterion for these parameters is left for future research. Note that the Mat\'ern kernel with $\nu\geq3/2$ is continuously differentiable which ensures the fast decay of the Fourier coefficients. For $\nu<3/2$ the Mat\'ern kernel is no longer continuously differentiable. However, for $\nu=1/2$ it coincides with the Laplacian kernel and for $\nu\neq(2n+1)/2$, $n\in\Z_{\geq0}$ the computation of the Mat\'ern kernel involves the evaluation of modified Bessel functions which is computationally costly such that those kernels are rarely used in practice.
We report the run time and the per-summand error $\frac{\|s_\text{true}-s_\text{approx}\|_1}{M\sum_{n=1}^N|w_n|}$.

\begin{figure}[!h]
\centering
\begin{subfigure}{.3\textwidth}
\includegraphics[width=\textwidth]{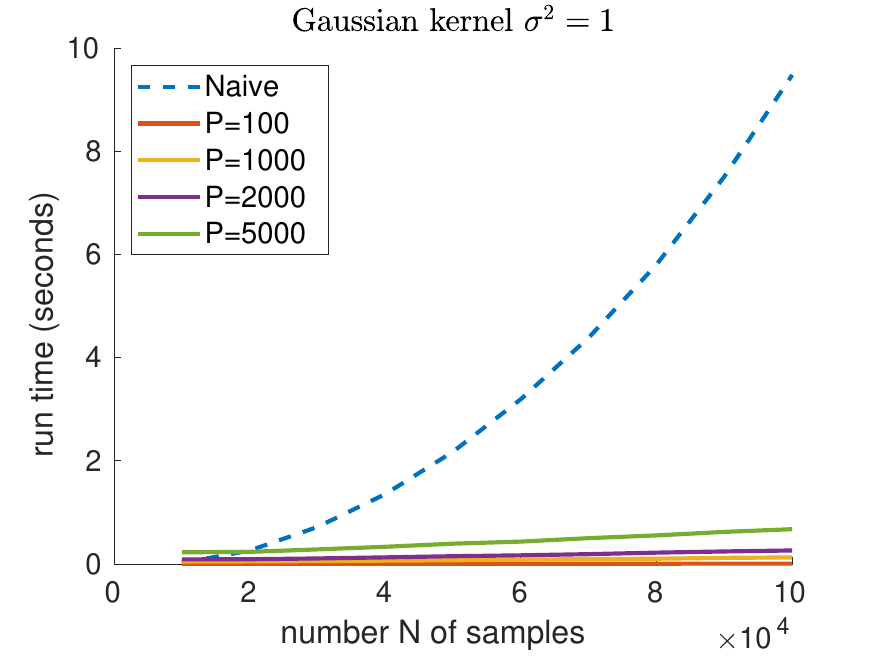}
\end{subfigure}
\hspace{.2cm}
\begin{subfigure}{.3\textwidth}
\includegraphics[width=\textwidth]{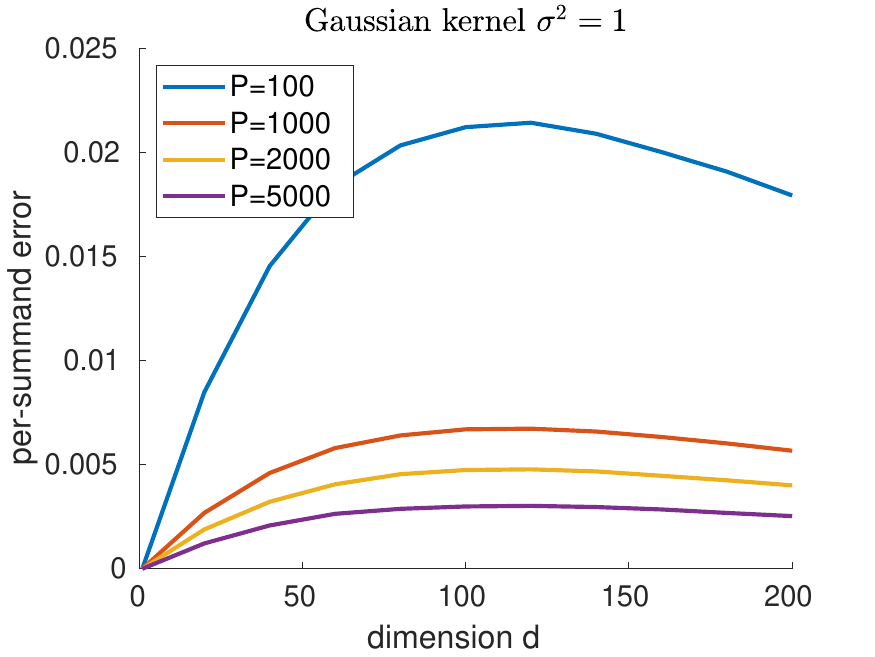}
\end{subfigure}
\hspace{.2cm}
\begin{subfigure}{.3\textwidth}
\includegraphics[width=\textwidth]{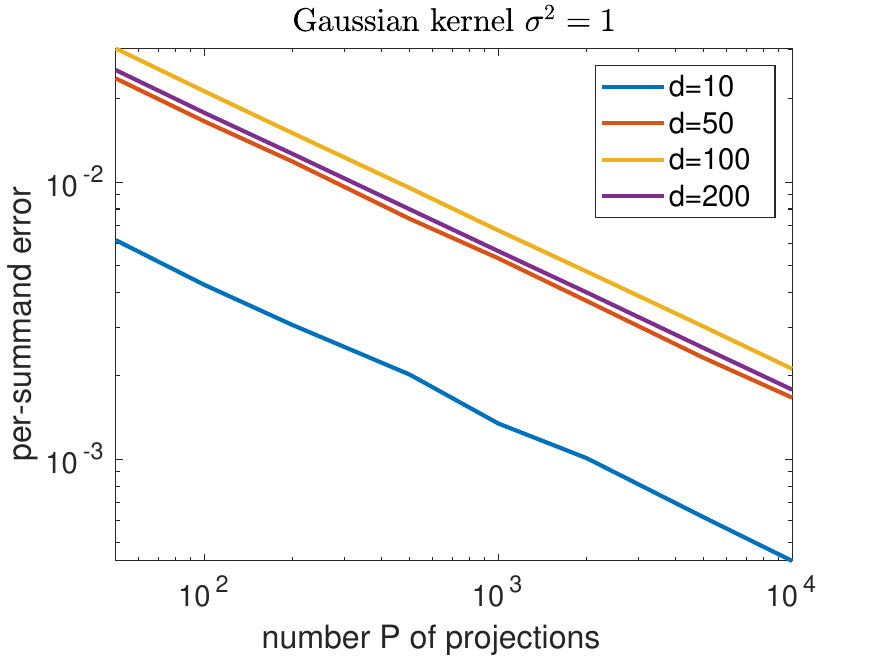}
\end{subfigure}

\begin{subfigure}{.3\textwidth}
\includegraphics[width=\textwidth]{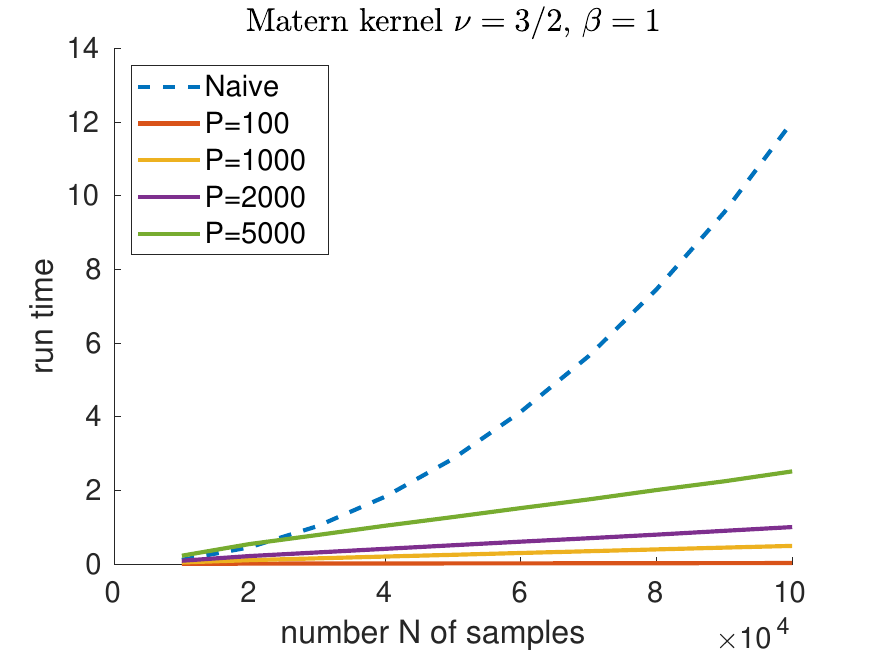}
\end{subfigure}
\hspace{.2cm}
\begin{subfigure}{.3\textwidth}
\includegraphics[width=\textwidth]{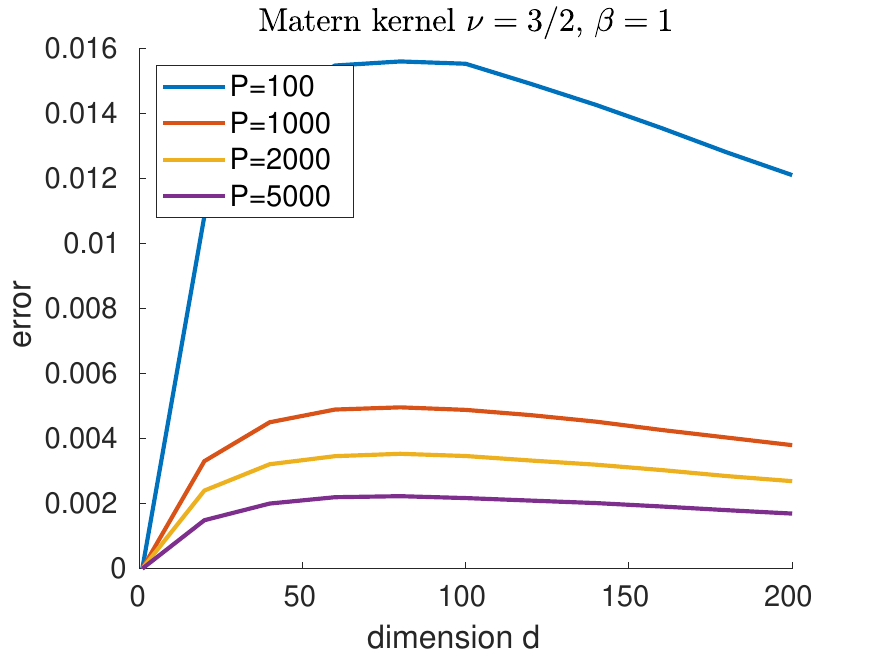}
\end{subfigure}
\hspace{.2cm}
\begin{subfigure}{.3\textwidth}
\includegraphics[width=\textwidth]{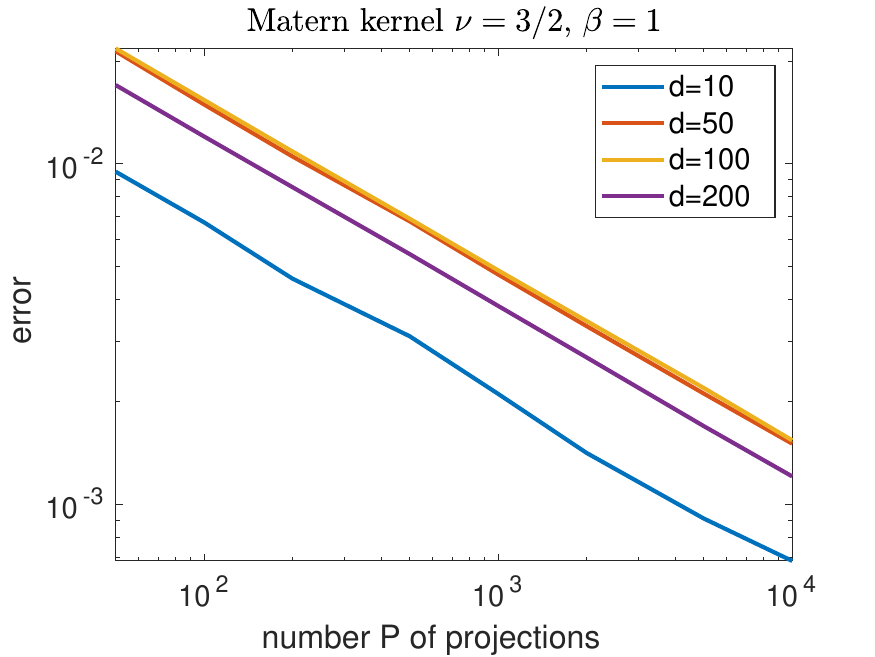}
\end{subfigure}

\begin{subfigure}{.3\textwidth}
\includegraphics[width=\textwidth]{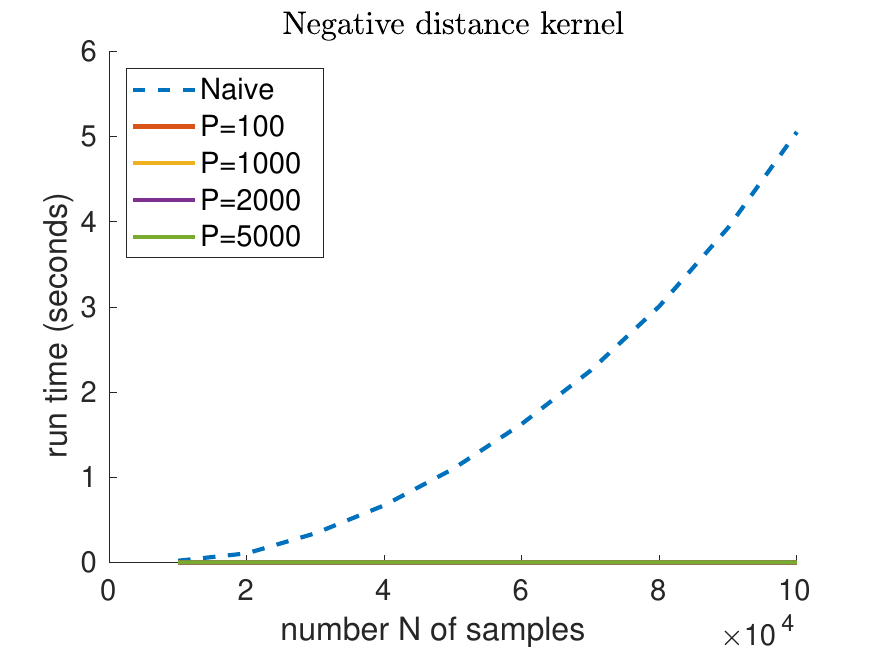}
\end{subfigure}
\hspace{.2cm}
\begin{subfigure}{.3\textwidth}
\includegraphics[width=\textwidth]{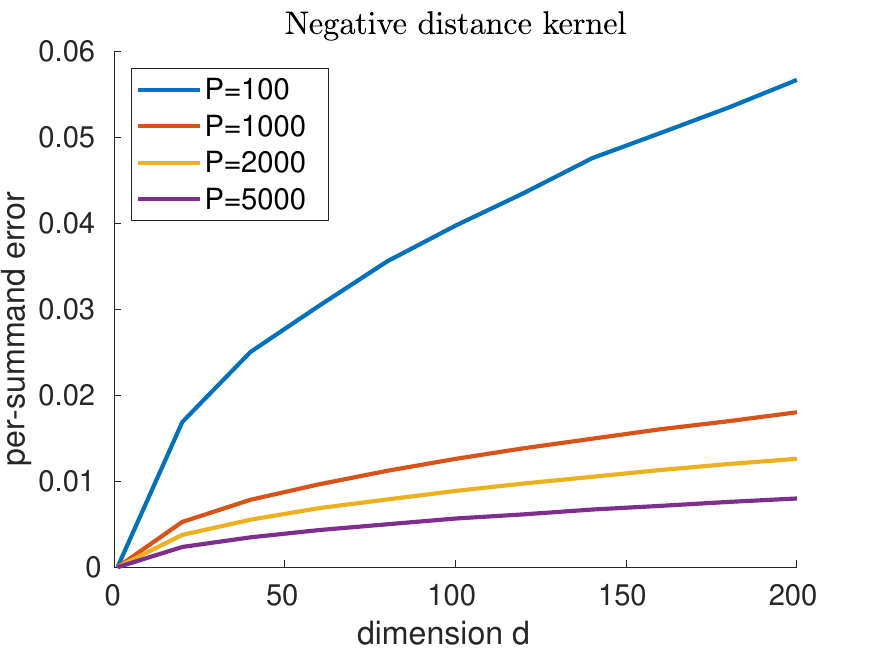}
\end{subfigure}
\hspace{.2cm}
\begin{subfigure}{.3\textwidth}
\includegraphics[width=\textwidth]{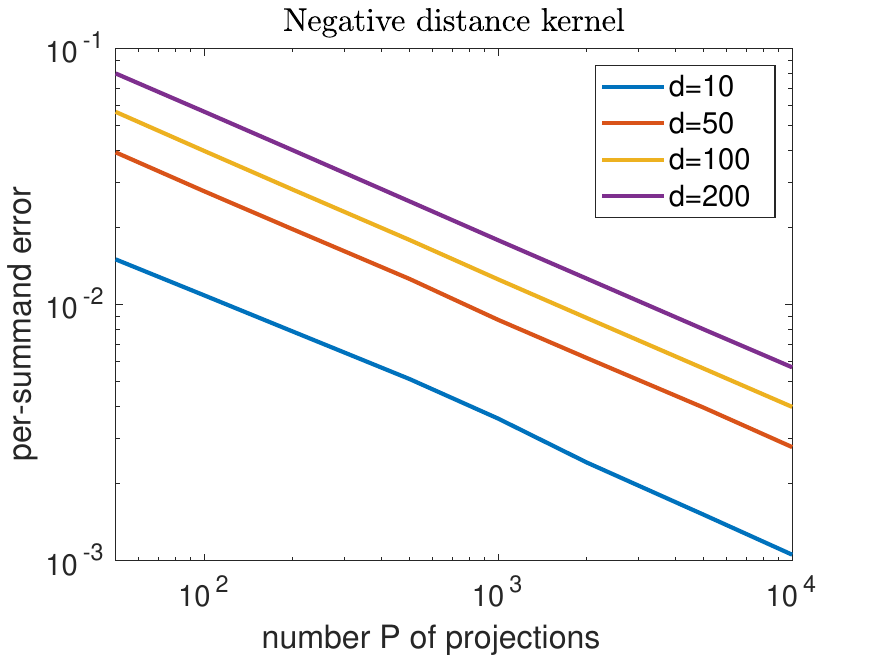}
\end{subfigure}

\begin{subfigure}{.3\textwidth}
\includegraphics[width=\textwidth]{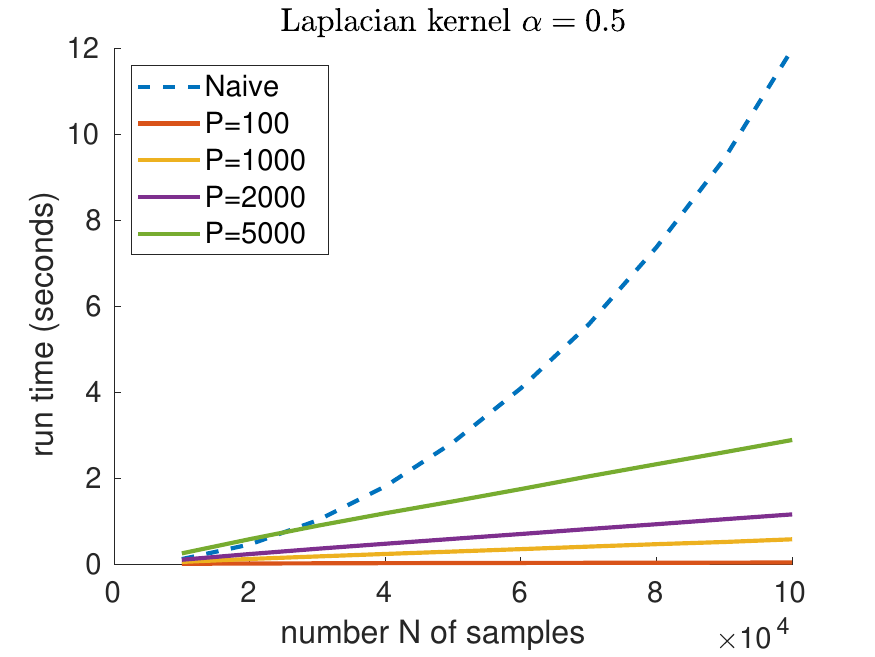}
\caption*{\scriptsize Run time dependence on $N$ for $d=50$ and several values of $P$.}
\end{subfigure}
\hspace{.2cm}
\begin{subfigure}{.3\textwidth}
\includegraphics[width=\textwidth]{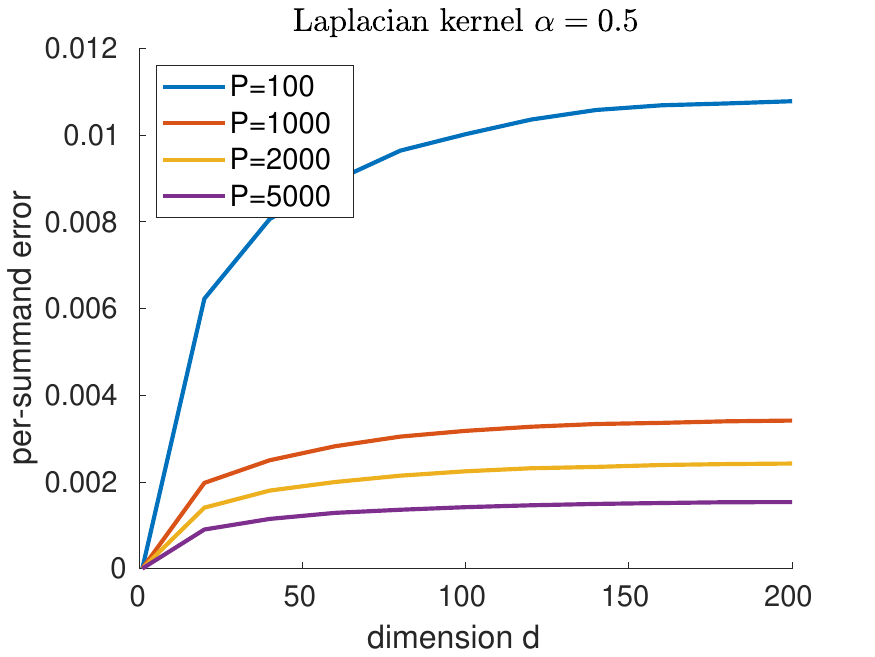}
\caption*{\scriptsize Error dependence on $d$ for $N=10000$ and several values of $P$.}
\end{subfigure}
\hspace{.2cm}
\begin{subfigure}{.3\textwidth}
\includegraphics[width=\textwidth]{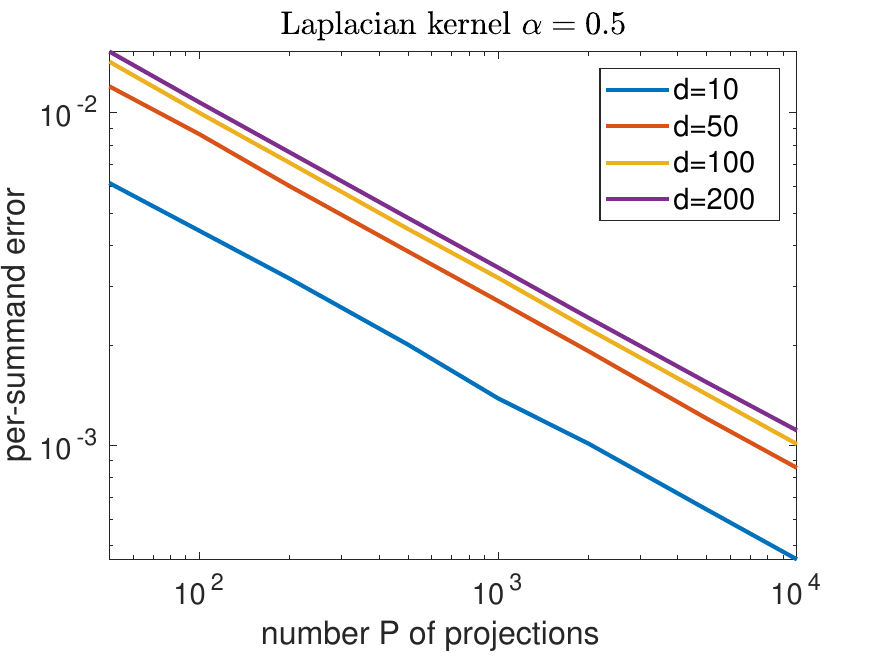}
\caption*{\scriptsize Error dependence on $P$ for $N=10000$ and several values of $d$.}
\end{subfigure}
\caption{Dependence of run time and the per-summand error of the fast kernel summation on the number $N$ of samples, the number $P$ of projections and the dimension $d$ for different kinds of kernels. From top to bottom: Gaussian kernel, Mat\'ern kernel, negative distance kernel, Laplacian kernel.}
\label{fig:comparison_fast_summation}
\end{figure}

\subsection{Results}
The code of this experiment is written in PyTorch and runs on a single NVIDIA RTX 4090 GPU with 24 GB memory.
We visualize the results in Figure~\ref{fig:comparison_fast_summation} as follows:
The plot on the left shows the dependence of the run time on $N$ for the naive implementation and the fast kernel summation. The results verify the quadratic dependence on $N$ for the naive implementation and the linear dependence on $N$ for the proposed fast kernel summation. In particular, we obtain a significant speed-up for large $N$.
For this plot, we fix the dimension to $d=50$. However, the naive and fast implementations depend both linearly on the dimension. Therefore, the conclusions are exactly the same for other choices of $d$.
In the middle, we report the dependence of the error on the dimension $d$ of the problem. We can see, that for the Gaussian, Mat\'ern and Laplacian the error saturates at a certain dimension. In particular, this verifies the dimension-independent error bound from Theorem~\ref{cor:error_bound}.
For the negative distance kernel we can see the square-root dependence on the dimension which again verifies Theorem~\ref{cor:error_bound} numerically.
On the right, we visualize the dependence of the error on the number $P$ of projections in a log-log plot. 
The results verify the asymptotic behaviour of $O(P^{-1/2})$ from Theorem~\ref{cor:error_bound} and Proposition~\ref{thm:accuracy_P}.

\subsection{Comparison to Random Fourier Features}\label{sec:cmp_RFF}
Finally, we compare our fast summation method with random Fourier features (RFF) proposed in \cite{RR2007} and further analyzed, e.g., in \cite{B2017,HSSTTW2021,LTOS2021}. To this end, we give a short summary on RFF, outline the differences to the proposed method and perform a numerical comparison for the Gaussian and the Laplacian kernel.

\paragraph{Fast Summation by Random Fourier Features}
Let $K\colon\R^d\times\R^d\to\R$ be given by $K(x,y)=\Psi(x-y)$ be a bounded shift-invariant positive definite kernel with $K(x,x)=\Psi(0)=1$. 
Then it holds by Bochners' theorem \cite{B1933} that there exists a probability measure $\mu$ on $\R^d$ such that 
\begin{equation*}
\Psi(x-y)=\hat\mu(x-y)=\E_{v\sim\mu}[\exp(2\pi i \langle x-y, v\rangle)]=\E_{v\sim\mu}[\exp(2\pi i \langle x,v\rangle)\exp(-2\pi i \langle y, v \rangle)].
\end{equation*}
By taking the real part on both sides we arrive at
\begin{equation*}
\Psi(x-y)=\E_{v\sim\mu}[\cos(\langle x, v \rangle)\cos(\langle y, v \rangle)+\sin(\langle x, v \rangle)\sin(\langle y, v \rangle)]
\end{equation*}
In many applications this is reformulated by some trigonometric calculations computations as
\begin{equation*}
\Psi(x-y)=2\E_{b\sim\mathcal U_{[0,2\pi)}}\E_{v\sim\mu}[\cos(\langle x, v \rangle+b)\cos(\langle y, v \rangle+b)],
\end{equation*}
see \cite{RR2007} for details.
Now, the basic idea of RFF is to subsample the expectations. More precisely let $v_1,...,v_D$ be iid samples from $\mu$ and $b_1,...,b_D$ be iid samples uniformly drawn from $[0,2\pi)$. Then it holds
\begin{equation*}
K(x,y)=\Psi(x-y)\approx 2\sum_{p=1}^D\cos(\langle x, v_p \rangle+b_p)\cos(\langle y, v_p \rangle+b_p)
\end{equation*}
or
\begin{equation*}
K(x,y)=\Psi(x-y)\approx \sum_{p=1}^D\cos(\langle x, v_p \rangle)\cos(\langle y, v_p \rangle)+\sin(\langle x, v_p \rangle)\sin(\langle y, v_p \rangle)
\end{equation*}
Consequently, we can approximate the Fourier sums $s_m$ as
\begin{equation}\label{eq:mit_b}
s_m=\sum_{n=1}^N w_nK(x_n,y_m)\approx 2\sum_{p=1}^D \cos(\langle y_m, v_p \rangle+b_p)\sum_{n=1}^N w_n\cos(\langle x_n, v_p \rangle+b_p)
\end{equation}
or
\begin{equation}\label{eq:ohne_b}
\begin{aligned}
s_m=\sum_{n=1}^N w_nK(x_n,y_m)&\approx \sum_{p=1}^D \cos(\langle y_m, v_p \rangle)\sum_{n=1}^N w_n\cos(\langle x_n, v_p \rangle)\\&\quad+\sum_{p=1}^D\sin(\langle y_m, v_p \rangle)\sum_{n=1}^N w_n\sin(\langle x_n, v_p \rangle)
\end{aligned}
\end{equation}
which can be computed in $O(Dd(N+M))$.
The measure $\mu$ is given by a standard normal distribution for the Gaussian kernel, by a multivariate Cauchy distribution for the Laplacian kernel and a multivariate Student-t distribution for the Mat\'ern kernel, see e.g.~\cite[(4.15)]{R2006}.
From a theoretical side, we see the following two main advantages of the slicing approach over RFF.
\begin{itemize}
\item[-] \textbf{RFF are limited to positive definite kernels}. In particular, RFF are not applicable for kernels which are only conditionally positive definite since Bochners' theorem does not hold true in this case. Such kernels include the thin plate spline kernel $K(x,y)=\|x-y\|^2\log(\|x-y\|)$ \cite{D1975} and the negative distance kernel $K(x,y)=-\|x-y\|$. Especially the latter one is heavily used in practice since the corresponding MMD leads to the energy distance \cite{GBRSS2006,szekely2002}.
Applications include halftoning of images \cite{TSGSW2011}, generative modelling via MMD flows \cite{HWAH2023} and posterior sampling in imaging inverse problems \cite{HHABCS2023}.
In particular, for the negative distance kernel, we have seen in the previous subsection that the slicing approach is very fast.
On the opposite side, RFF can handle non-radial shift-invariant kernels like the $L^1$-Laplacian kernel $K(x,y)=\exp(-\|x-y\|_1)$ which we have not studied in the context of the slicing approach.
\item[-] The approximated kernel by RFFs \textbf{is always smooth} even when the original kernel was non-smooth. Empirically, this leads to poor approximation qualities for non-smooth kernels including the Laplacian, see also \cite{CJWW2023} for some negative theoretical results in this direction.
In contrast, the sliced approximation of non-smooth kernels can be non-smooth as well. Indeed, we will see in the numerical part that slicing significantly outperforms RFFs for the Laplacian kernel.
\end{itemize}

\paragraph{Numerical Comparison}

\begin{figure}
\centering
\begin{subfigure}{.3\textwidth}
\includegraphics[width=\textwidth]{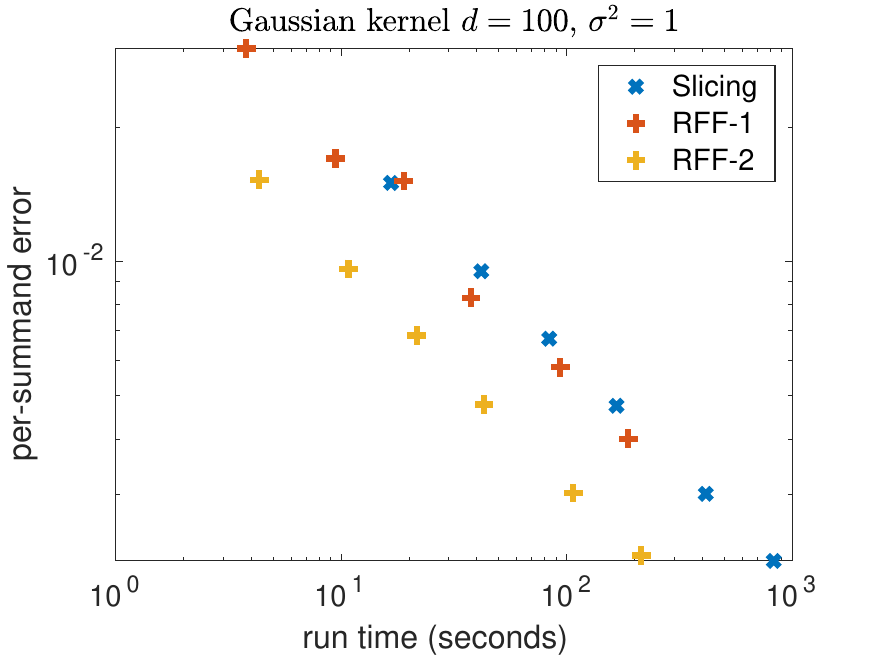}
\end{subfigure}
\hspace{.2cm}
\begin{subfigure}{.3\textwidth}
\includegraphics[width=\textwidth]{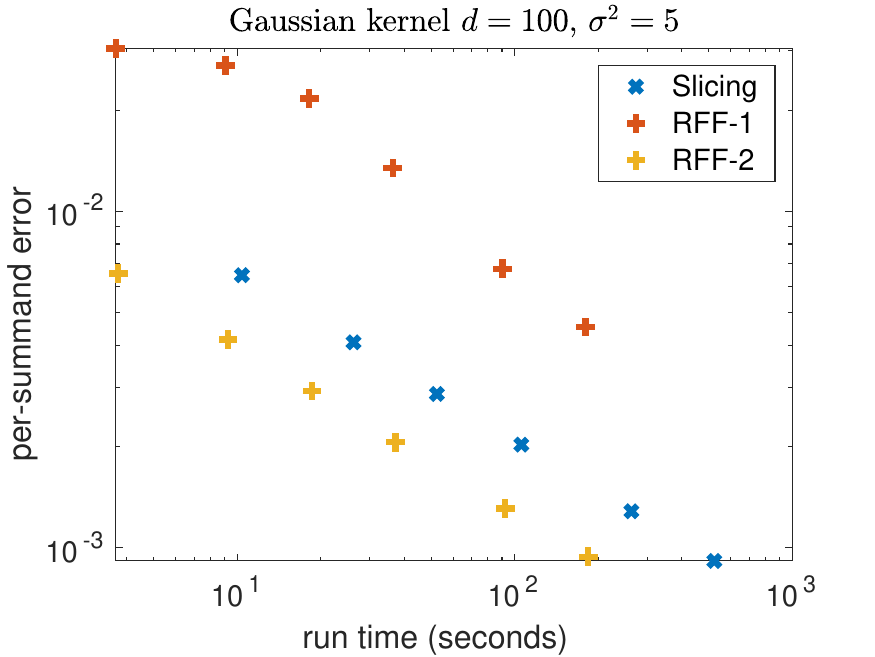}
\end{subfigure}
\hspace{.2cm}
\begin{subfigure}{.3\textwidth}
\includegraphics[width=\textwidth]{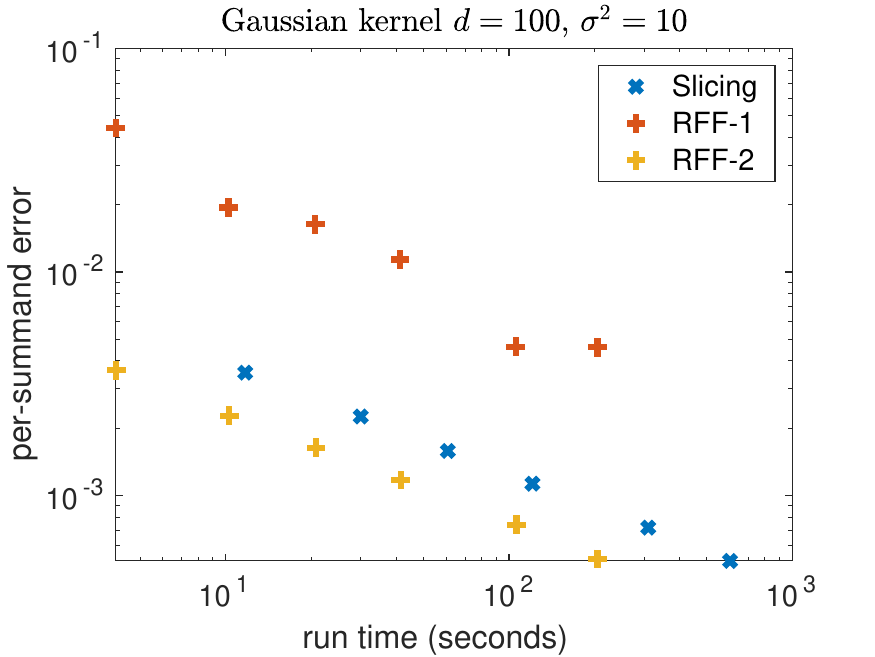}
\end{subfigure}

\begin{subfigure}{.3\textwidth}
\includegraphics[width=\textwidth]{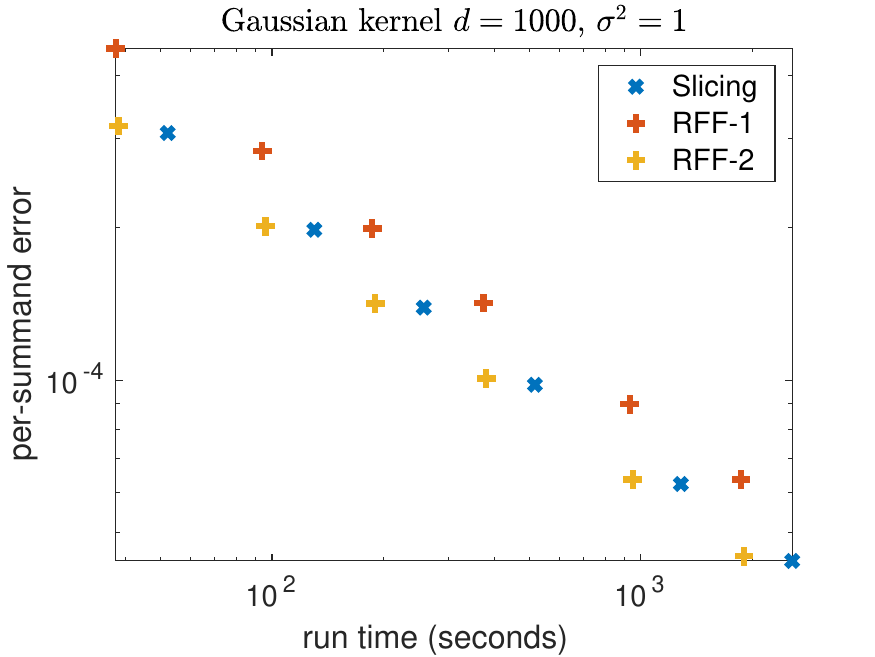}
\end{subfigure}
\hspace{.2cm}
\begin{subfigure}{.3\textwidth}
\includegraphics[width=\textwidth]{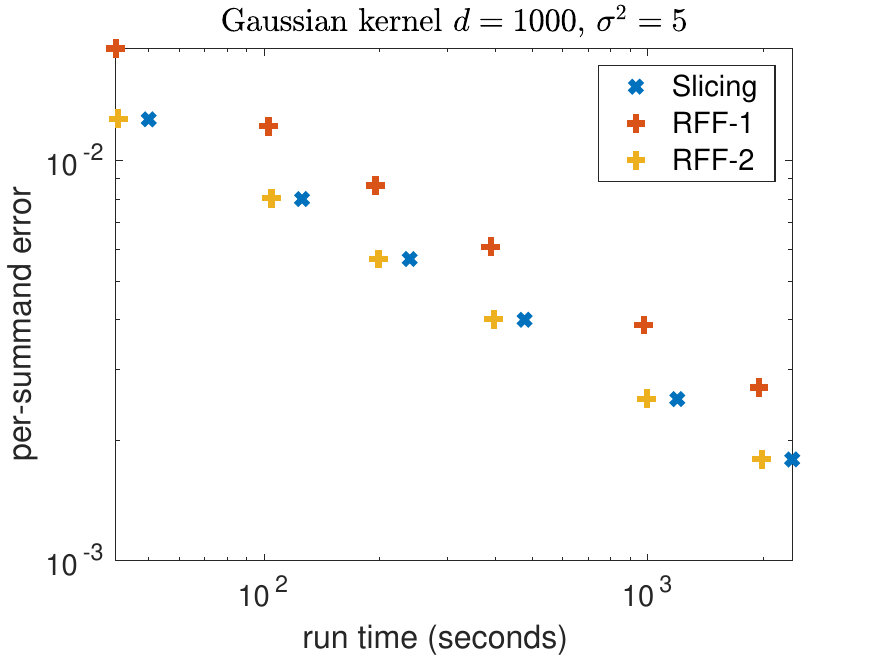}
\end{subfigure}
\hspace{.2cm}
\begin{subfigure}{.3\textwidth}
\includegraphics[width=\textwidth]{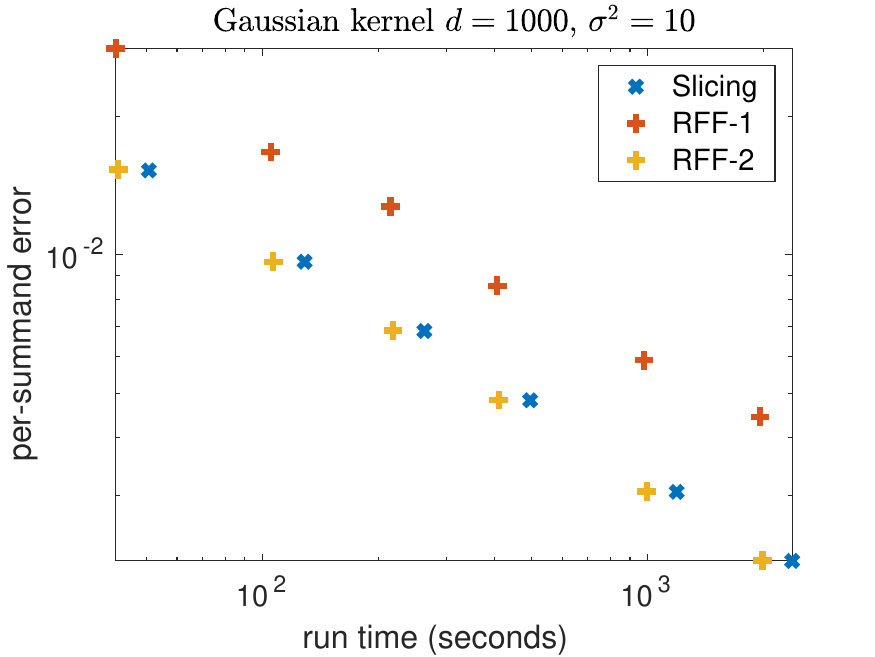}
\end{subfigure}

\begin{subfigure}{.3\textwidth}
\includegraphics[width=\textwidth]{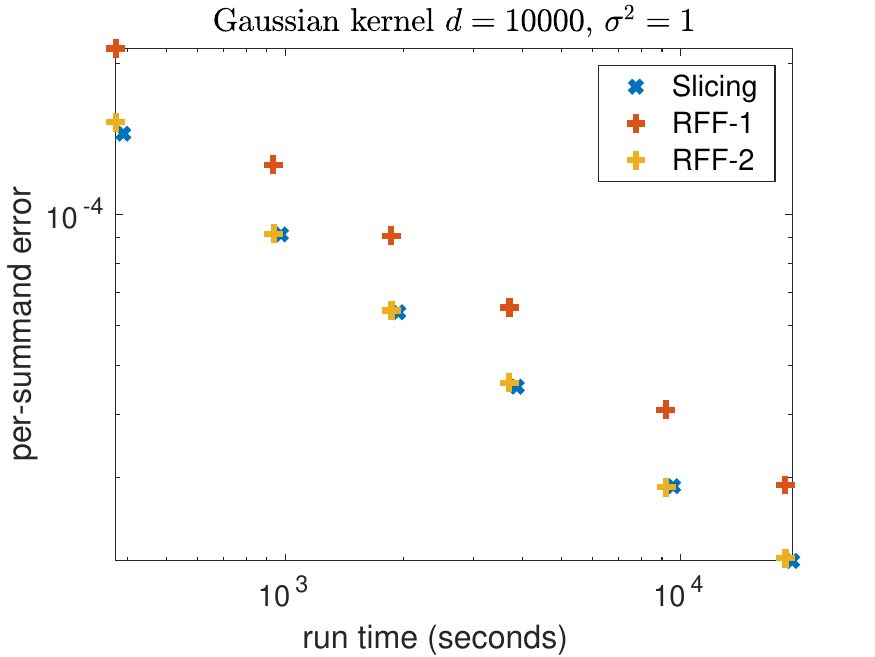}
\end{subfigure}
\hspace{.2cm}
\begin{subfigure}{.3\textwidth}
\includegraphics[width=\textwidth]{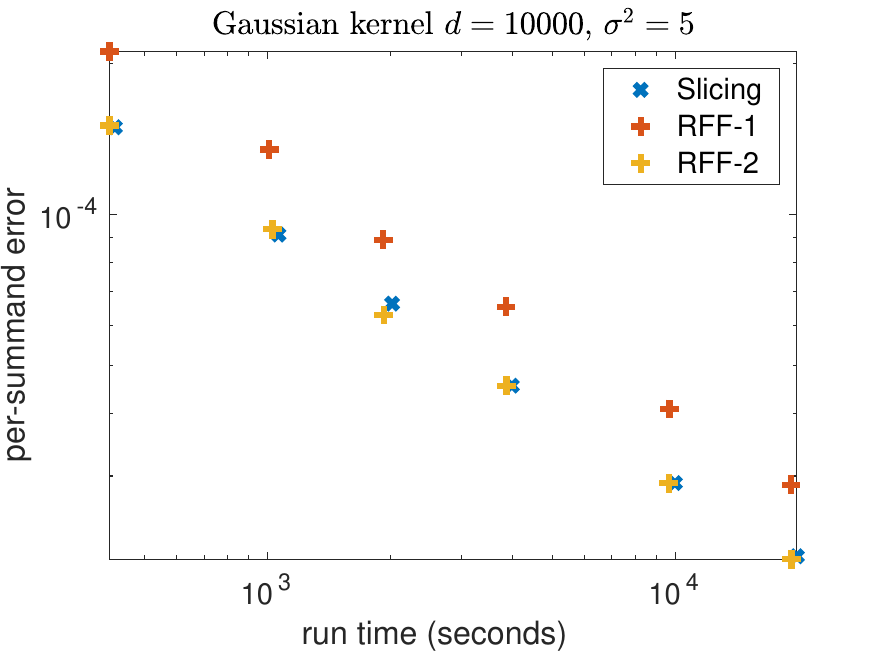}
\end{subfigure}
\hspace{.2cm}
\begin{subfigure}{.3\textwidth}
\includegraphics[width=\textwidth]{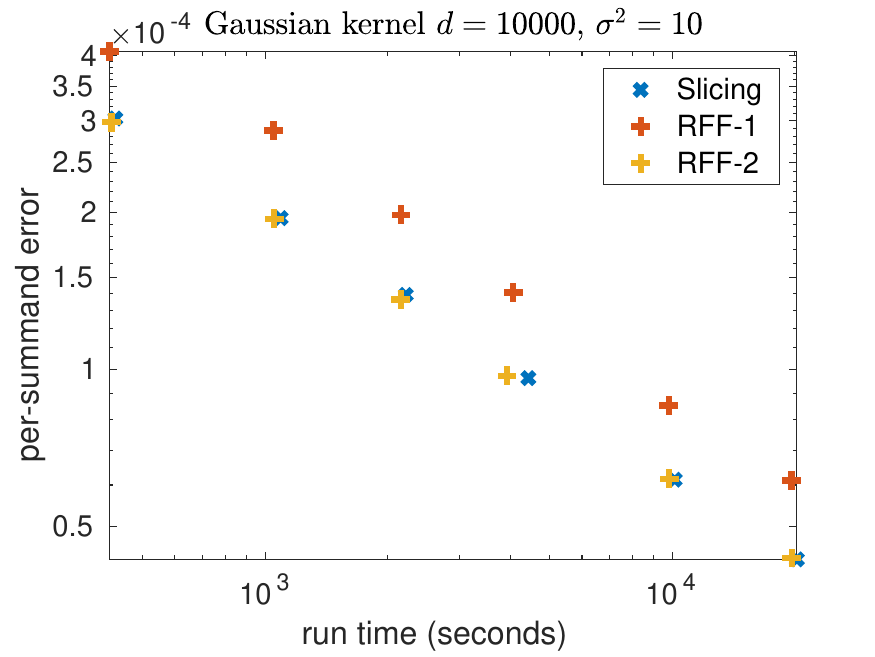}
\end{subfigure}

\begin{subfigure}{.3\textwidth}
\includegraphics[width=\textwidth]{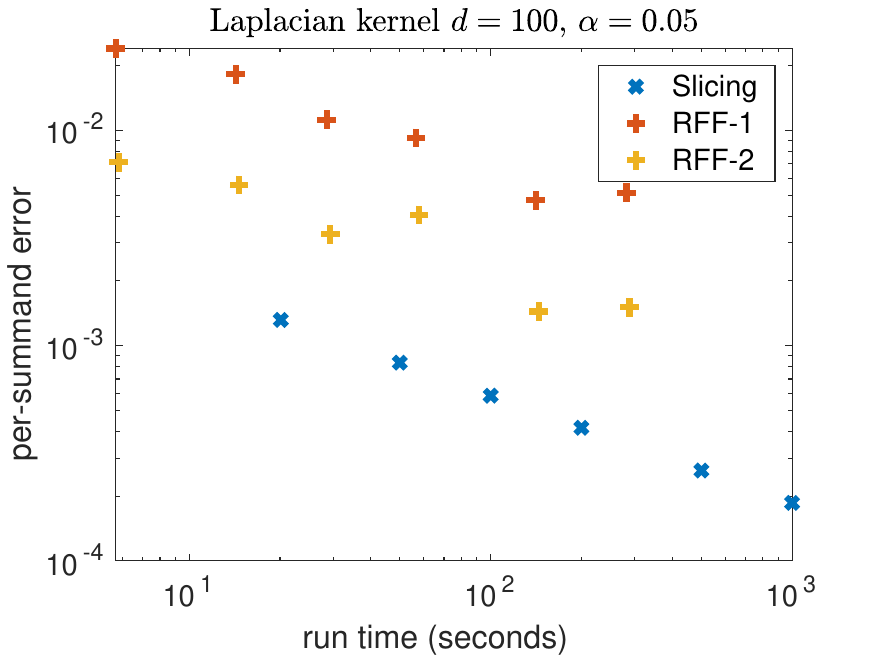}
\end{subfigure}
\hspace{.2cm}
\begin{subfigure}{.3\textwidth}
\includegraphics[width=\textwidth]{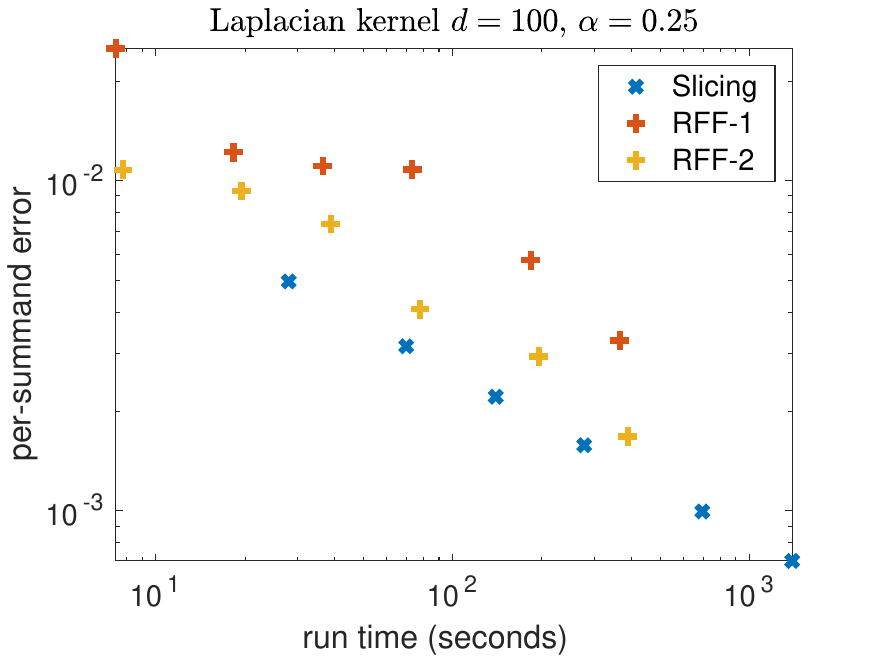}
\end{subfigure}
\hspace{.2cm}
\begin{subfigure}{.3\textwidth}
\includegraphics[width=\textwidth]{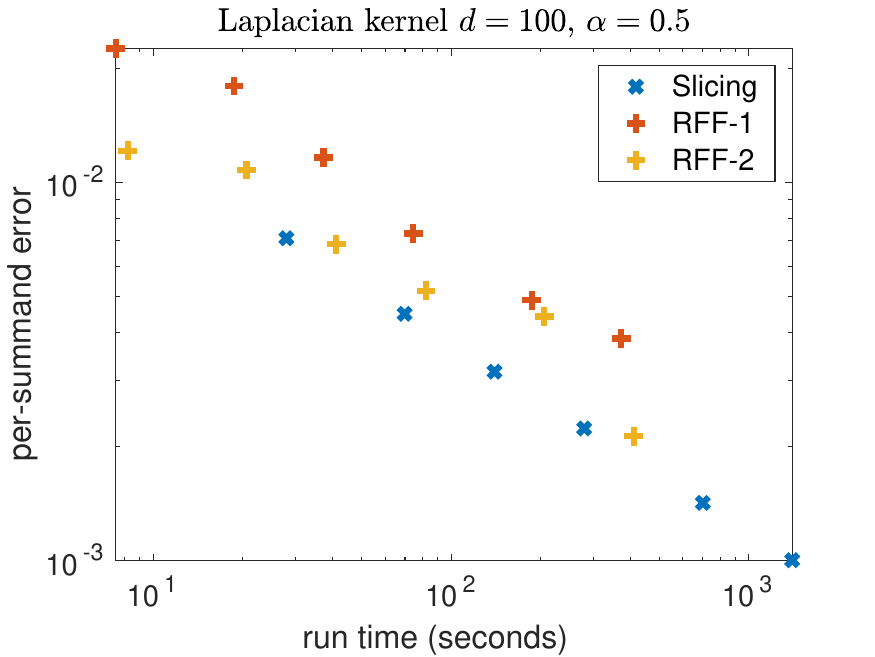}
\end{subfigure}

\begin{subfigure}{.3\textwidth}
\includegraphics[width=\textwidth]{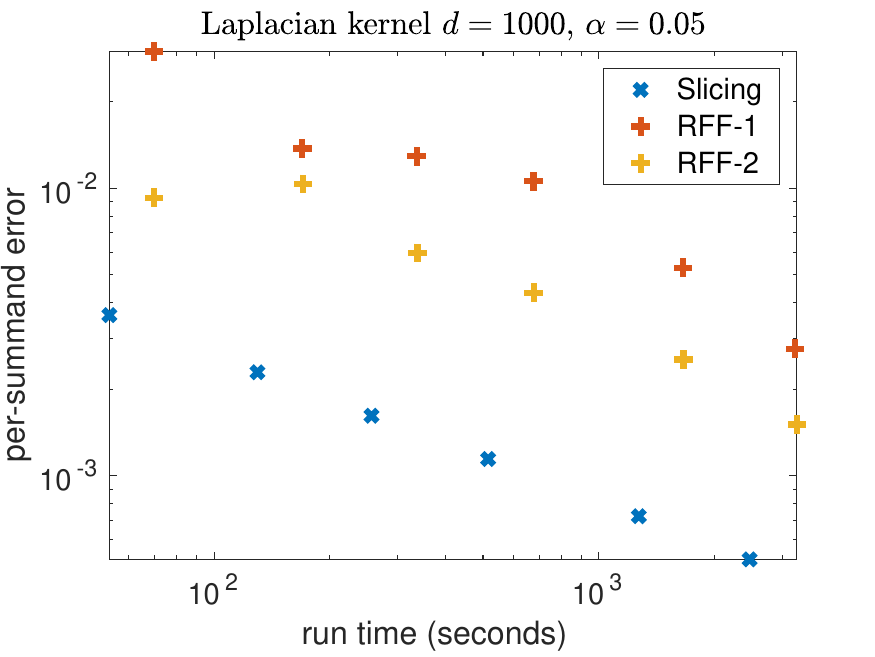}
\end{subfigure}
\hspace{.2cm}
\begin{subfigure}{.3\textwidth}
\includegraphics[width=\textwidth]{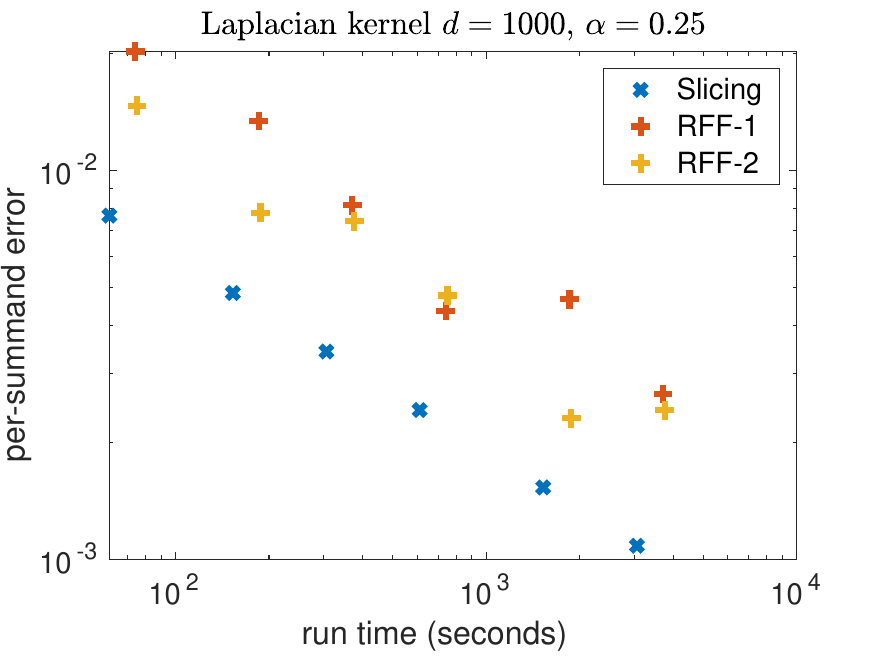}
\end{subfigure}
\hspace{.2cm}
\begin{subfigure}{.3\textwidth}
\includegraphics[width=\textwidth]{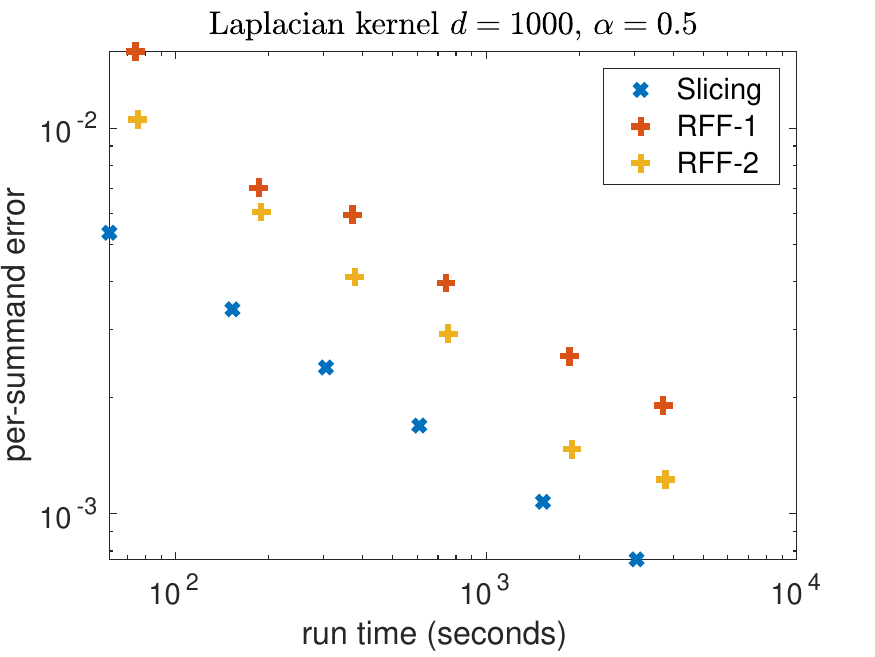}
\end{subfigure}

\begin{subfigure}{.3\textwidth}
\includegraphics[width=\textwidth]{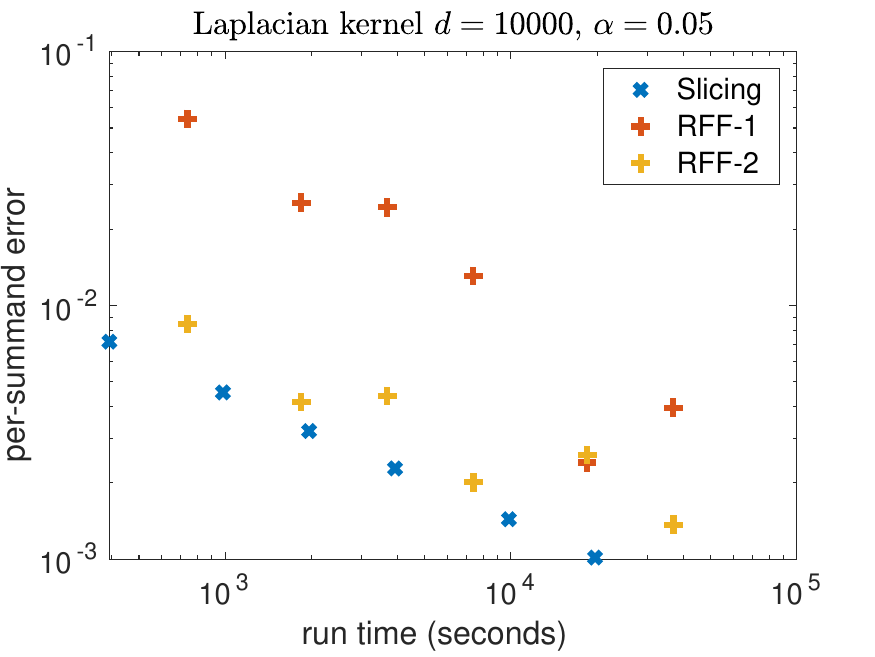}
\end{subfigure}
\hspace{.2cm}
\begin{subfigure}{.3\textwidth}
\includegraphics[width=\textwidth]{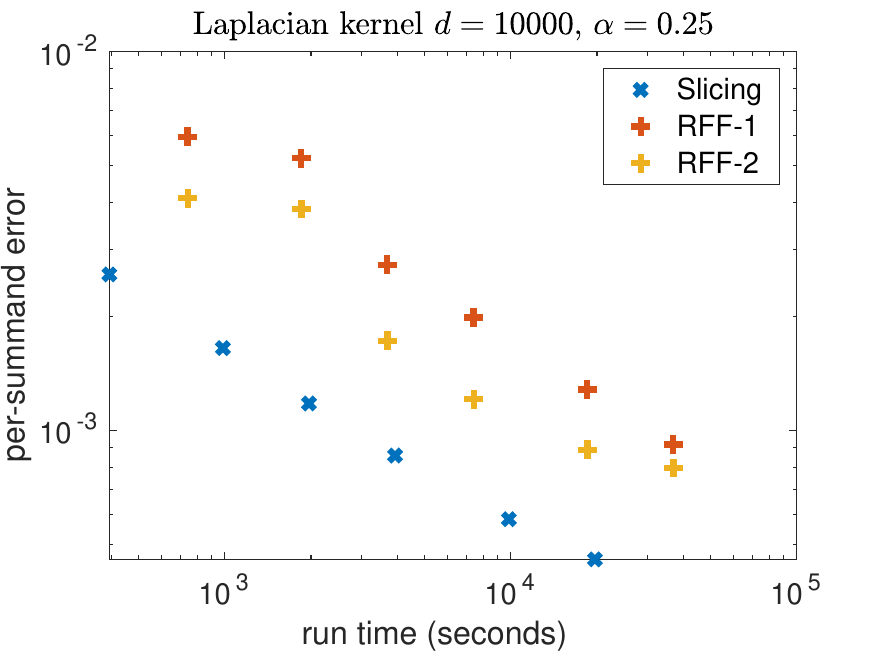}
\end{subfigure}
\hspace{.2cm}
\begin{subfigure}{.3\textwidth}
\includegraphics[width=\textwidth]{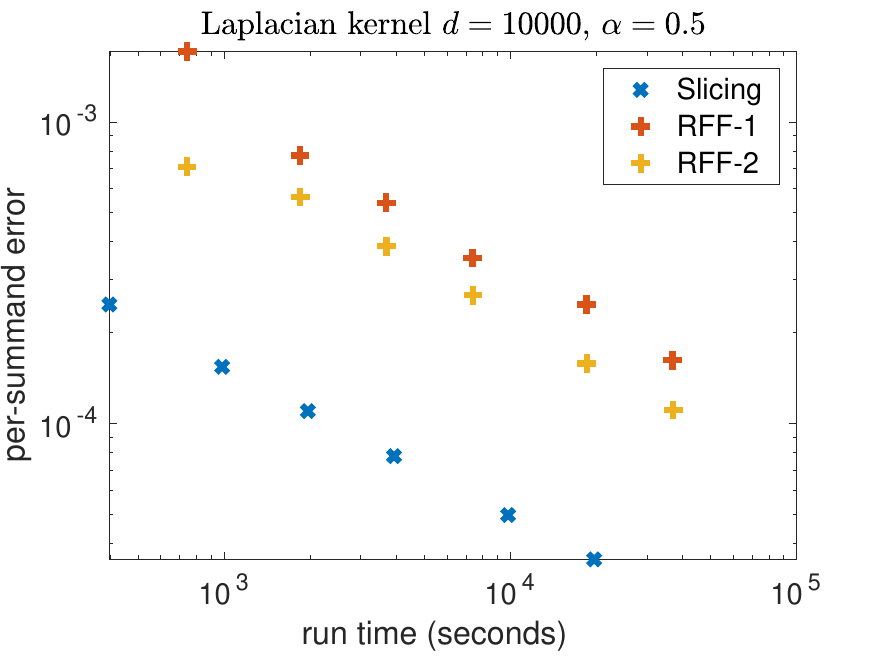}
\end{subfigure}
\caption{Comparison of the error versus run time for Slicing and RFF for different dimensions and kernel parameters. Top three rows: Gaussian kernel, bottom three rows: Laplacian kernel.}
\label{fig:cmp_RFF}
\end{figure}

\begin{table}
\centering
Gaussian kernel, $\sigma^2=5$, $d=1000$

\scalebox{.54}{
\begin{tabular}{c|cccccc}
Error&$P=D=200$&$P=D=500$&$P=D=1000$&$P=D=2000$&$P=D=5000$&$P=D=10000$\\\hline
Slicing&$1.27\cdot10^{-2}\pm3.83\cdot10^{-5}$&$8.02\cdot10^{-3}\pm2.16\cdot10^{-5}$&$5.68\cdot10^{-3}\pm9.49\cdot10^{-6}$&$4.00\cdot10^{-3}\pm5.18\cdot10^{-6}$&$2.54\cdot10^{-3}\pm5.55\cdot10^{-6}$&$1.79\cdot10^{-3}\pm5.72\cdot10^{-6}$\\
RFF-1&$1.91\cdot10^{-2}\pm7.04\cdot10^{-4}$&$1.22\cdot10^{-2}\pm2.87\cdot10^{-4}$&$8.66\cdot10^{-3}\pm3.49\cdot10^{-4}$&$6.09\cdot10^{-3}\pm2.90\cdot10^{-4}$&$3.89\cdot10^{-3}\pm2.06\cdot10^{-4}$&$2.71\cdot10^{-3}\pm7.30\cdot10^{-5}$\\
RFF-2&$1.27\cdot10^{-2}\pm6.06\cdot10^{-5}$&$8.05\cdot10^{-3}\pm3.11\cdot10^{-5}$&$5.68\cdot10^{-3}\pm1.64\cdot10^{-5}$&$4.01\cdot10^{-3}\pm8.21\cdot10^{-6}$&$2.54\cdot10^{-3}\pm8.00\cdot10^{-6}$&$1.80\cdot10^{-3}\pm5.20\cdot10^{-6}$
\end{tabular}}
\scalebox{.54}{
\begin{tabular}{c|cccccc}
Run time&$P=D=200$&$P=D=500$&$P=D=1000$&$P=D=2000$&$P=D=5000$&$P=D=10000$\\\hline
Slicing&$4.99\cdot10^{+1}\pm2.63\cdot10^{-1}$&$1.25\cdot10^{+2}\pm1.48\cdot10^{+0}$&$2.39\cdot10^{+2}\pm2.32\cdot10^{+0}$&$4.77\cdot10^{+2}\pm5.28\cdot10^{-2}$&$1.19\cdot10^{+3}\pm6.35\cdot10^{+0}$&$2.37\cdot10^{+3}\pm6.38\cdot10^{+0}$\\
RFF-1&$4.10\cdot10^{+1}\pm9.31\cdot10^{-2}$&$1.03\cdot10^{+2}\pm1.33\cdot10^{+0}$&$1.95\cdot10^{+2}\pm3.85\cdot10^{-1}$&$3.90\cdot10^{+2}\pm4.41\cdot10^{-2}$&$9.76\cdot10^{+2}\pm3.62\cdot10^{+0}$&$1.95\cdot10^{+3}\pm6.05\cdot10^{-1}$\\
RFF-2&$4.16\cdot10^{+1}\pm1.14\cdot10^{-1}$&$1.04\cdot10^{+2}\pm1.22\cdot10^{+0}$&$1.98\cdot10^{+2}\pm9.65\cdot10^{-2}$&$3.96\cdot10^{+2}\pm8.16\cdot10^{-2}$&$9.93\cdot10^{+2}\pm6.11\cdot10^{+0}$&$1.98\cdot10^{+3}\pm4.60\cdot10^{-1}$
\end{tabular}}
\vspace{.5cm}

Laplacian kernel, $\alpha=0.25$, $d=1000$

\scalebox{.54}{
\begin{tabular}{c|cccccc}
Error&$P=D=200$&$P=D=500$&$P=D=1000$&$P=D=2000$&$P=D=5000$&$P=D=10000$\\\hline
Slicing&$7.66\cdot10^{-3}\pm2.04\cdot10^{-5}$&$4.84\cdot10^{-3}\pm1.56\cdot10^{-5}$&$3.42\cdot10^{-3}\pm8.01\cdot10^{-6}$&$2.42\cdot10^{-3}\pm6.99\cdot10^{-6}$&$1.53\cdot10^{-3}\pm2.85\cdot10^{-6}$&$1.08\cdot10^{-3}\pm2.93\cdot10^{-6}$\\
RFF-1&$2.02\cdot10^{-2}\pm1.40\cdot10^{-2}$&$1.34\cdot10^{-2}\pm4.98\cdot10^{-3}$&$8.15\cdot10^{-3}\pm2.89\cdot10^{-3}$&$4.35\cdot10^{-3}\pm1.11\cdot10^{-3}$&$4.67\cdot10^{-3}\pm2.31\cdot10^{-3}$&$2.66\cdot10^{-3}\pm9.99\cdot10^{-4}$\\
RFF-2&$1.47\cdot10^{-2}\pm6.06\cdot10^{-5}$&$7.79\cdot10^{-3}\pm3.11\cdot10^{-5}$&$7.41\cdot10^{-3}\pm1.64\cdot10^{-5}$&$4.77\cdot10^{-3}\pm8.21\cdot10^{-6}$&$2.31\cdot10^{-3}\pm8.00\cdot10^{-6}$&$2.41\cdot10^{-3}\pm5.20\cdot10^{-6}$
\end{tabular}}
\scalebox{.54}{
\begin{tabular}{c|cccccc}
Run time&$P=D=200$&$P=D=500$&$P=D=1000$&$P=D=2000$&$P=D=5000$&$P=D=10000$\\\hline
Slicing&$6.12\cdot10^{+1}\pm3.66\cdot10^{-1}$&$1.53\cdot10^{+2}\pm5.87\cdot10^{-2}$&$3.06\cdot10^{+2}\pm9.80\cdot10^{-2}$&$6.11\cdot10^{+2}\pm6.62\cdot10^{-1}$&$1.53\cdot10^{+3}\pm4.95\cdot10^{-1}$&$3.06\cdot10^{+3}\pm4.09\cdot10^{+0}$\\
RFF-1&$7.41\cdot10^{+1}\pm4.83\cdot10^{-2}$&$1.85\cdot10^{+2}\pm2.47\cdot10^{-2}$&$3.71\cdot10^{+2}\pm4.08\cdot10^{-2}$&$7.41\cdot10^{+2}\pm2.48\cdot10^{-1}$&$1.85\cdot10^{+3}\pm4.20\cdot10^{-1}$&$3.71\cdot10^{+3}\pm2.05\cdot10^{+0}$\\
RFF-2&$7.52\cdot10^{+1}\pm2.88\cdot10^{-2}$&$1.88\cdot10^{+2}\pm6.91\cdot10^{-2}$&$3.76\cdot10^{+2}\pm1.02\cdot10^{-1}$&$7.52\cdot10^{+2}\pm2.08\cdot10^{-1}$&$1.88\cdot10^{+3}\pm2.60\cdot10^{-1}$&$3.76\cdot10^{+3}\pm1.65\cdot10^{+0}$
\end{tabular}}
\caption{Comparison of the error and run time for slicing and RFF for $d=1000$. For both the Gaussian and the Laplacian kernel, the upper table reports the errors and the bottom table reports the run times including standard deviations.}
\label{tab:RFF_vs_slicing}
\end{table}

Now we compare the performance of our slicing approach with RFF numerically. 
The code for this example is written in Julia and runs on a single-threaded CPU for reducing the dependence on the implementation details.
We denote the RFF approximation in \eqref{eq:mit_b} by RFF-1 and the RFF approximation in \eqref{eq:ohne_b} with RFF-2.
We use the same data generation process as before with $N=100000$ data points of dimension $d\in\{100,1000,10000\}$ and use $P\in\{200,500,1000,2000,5000,10000\}$ projections. In order to obtain a similar runtime, we choose $D=P$ Fourier features for the Gaussian and $D=2P$ for the Laplacian kernel.
In Figure~\ref{fig:cmp_RFF} we plot the runtime versus the error for the different choices of $d$  and different kernel parameters $\alpha\in\{0.05,0.25,0.5\}$ and $\sigma^2\in\{1,5,10\}$.
In addition, we report the precise values including standard deviations for $d=1000$, $\sigma^2=5$ and $\alpha=0.25$ in Table~\ref{tab:RFF_vs_slicing}.
Similarly as in \cite{SS2015}, we observe that the second variant (RFF-2) given by \eqref{eq:ohne_b} of RFF performs significantly better than the first variant (RFF-1) given by \eqref{eq:mit_b}.
For the Gaussian kernel, we can see that for $d=100$ and $d=1000$ RFF-1 performs slightly better than slicing. For $d=10000$ the performance of RFF-1 and slicing is similar.
For the Laplacian kernel we can see that slicing always performs better than both variants of RFF independent of the choice of $d$ or the kernel parameter.
In Table~\ref{tab:RFF_vs_slicing}, we can see that the errors of RFF-1 have a significantly higher standard deviation than the ones of RFF-2 and slicing. The standard deviation of the run time is negligible.

\subsection{Large Scale GPU Comparison}\label{sec:cmp_keops}

\begin{table}
\centering
Gaussian kernel, $\sigma^2=1$\vspace{.1cm}

\begin{tabular}{c|cccc}
$N$&$10^4$&$10^5$&$10^6$&$10^7$\\\hline
PyKeOps&$0.0145$ s&$0.289$ s&$22.4$ s&$2197$ s\\
Slicing $P=100$&$0.0065$ s&$0.014$ s&$0.279$ s& $2.72$ s\\
Slicing $P=1000$&$0.0493$ s&$0.144$ s&$2.80$ s& $25.4$ s\\
Slicing $P=2000$&$0.0953$ s&$0.283$ s&$5.70$ s& $50.9$ s\\
Slicing $P=5000$&$0.2365$ s&$0.713$ s&$13.1$ s&$127.2$ s
\end{tabular}\vspace{.5cm}

Negative distance kernel\vspace{.2cm}

\begin{tabular}{c|cccc}
$N$&$10^4$&$10^5$&$10^6$&$10^7$\\\hline
PyKeOps&$0.0137$ s&$0.286$ s&$22.4$ s&$2196$ s\\
Slicing $P=100$&$0.0021$ s&$0.009$ s&$0.070$ s& $1.51$ s\\
Slicing $P=1000$&$0.0092$ s&$0.093$ s&$0.686$ s& $15.1$ s\\
Slicing $P=2000$&$0.0192$ s&$0.183$ s&$1.37$ s& $30.17$ s\\
Slicing $P=5000$&$0.0455$ s&$0.456$ s&$3.42$ s&$75.58$ s
\end{tabular}
\caption{Run time comparison of PyKeOps \cite{CFGCD2021} with slicing for the Gaussian (top) and negative distance kernel (bottom).}
\label{tab:runtimes_keops}
\end{table}

Finally, we compare the run time of slicing with the brute-force kernel summation implemented by the KeOps library \cite{CFGCD2021}\footnote{We use the Python wrapper PyKeOps.} for a large number of data points. 
KeOps generally evaluates the full kernel sum such that the kernel summation has complexity $O(N^2)$.
However, it uses symbolic tensors and avoids memory transfers to speed up the computations massively.
We want to emphasize that from an implementations viewpoint this is not a fair comparison. In particular:
\begin{itemize}
\item[-] KeOps consists out of compiled code which natively written in C++/Cuda. In contrast, we use a high-level implementation in PyTorch such that the run time heavily depends on vectorization and the use of built-in functions.
\item[-] Similar techniques as used in the implementation of KeOps might be applicable to accelerate the slicing approach (or some subroutines of it) as well. 
\end{itemize}
However, discussing and providing high-performance GPU-optimized implementations is not within the scope of this paper.

This example is implemented with PyTorch and we run it on a single NVIDIA RTX4090 GPU with 24 GB memory.
We use the same data generation process as before with $N\in\{10^4,10^5,10^6,10^7\}$
data points of dimension $d=100$ and use $P\in\{100,1000,2000,5000\}$ projections.
As kernel we choose the Gaussian kernel with $\sigma^2=1$ and the negative distance kernel.
The resulting run times are given in Table~\ref{tab:runtimes_keops}.
Even though PyKeOps shows a competitive run time for up to $10^5$ data points, it of course cannot escape the quadratic complexity. Consequently for $N=10^6$ or larger, slicing is significantly faster than computing the naive kernel sum via PyKeOps.

\begin{remark}[Memory Requirements and Batching]
If all one-dimensional kernel summations are computed in parallel, the asymptotic memory requirements of the slicing algorithm are $O(Nd)$ for storing the data set plus $O(NP)$ for storing all projected data sets. If $N$ becomes large and $P>d$, this can be a limitation.
To overcome these memory requirements, we only consider a batch of $B<P$ projections at the same time and iterate this procedure until we have computed all projections.
This reduces the memory requirements from $O(N(d+P))$ to $O(N(d+B))$. 
However, it also leads to a poor parallelization of the algorithm whenever $B$ is too small.
We can see this effect for the negative distance kernel in Table~\ref{tab:runtimes_keops}, where we use a batch size of $100$ for $N=10^6$ and a batch size of $10$ for $N=10^7$. While the slicing algorithm has theoretically computational complexity of $O(N\log(N))$, the factor between the run times of $N=10^6$ and $N=10^7$ is worse than a factor of $10$.
If we reduce the batch size also for smaller values of $N$, we retain again the approximate factor of $10$ between the run times for $N=10^6$ and $N=10^7$.
\end{remark}

\section{Conclusions}\label{sec:concl}

We proposed a method to approximate large kernel sums, which appear in all kinds of kernel methods.
Our approach is based on two ideas.
First, we reduce the problem to the one-dimensional case by representing kernels as sliced kernel of a suitable one-dimensional counterpart.
In the arising one-dimensional setting we apply a fast Fourier summation method.
The resulting fast kernel summation method is fast, accurate and simple to implement.
In particular, it reduces the quadratic dependence of the run time on the number of samples to a linear dependence.
Numerical experiments verify the correctness of the claims and demonstrate advantages of slicing for fast kernel summations. In particular, slicing is not limited to positive definite kernels and has therefore a broader applicability than other methods like random Fourier features.

\paragraph{Limitations}
The error estimates from Proposition~\ref{thm:accuracy_P} and Theorem~\ref{cor:error_bound} are given in the terms of absolute errors and numerical computations suggest that such an estimate does not hold true relative to the absolute value of the kernel sum.
Moreover, the error estimates only consider the error in one kernel evaluation and neglects how these errors interact in the kernels sum. Depending on the distribution of the underlying dataset, finer error rates for the whole kernel sum might hold true.
Finally, the fast Fourier summation from Section~\ref{sec:1d} is limited to smooth kernels, see Remark~\ref{rem:smoothness}. For some kernels, we can overcome this problem. This includes the negative distance kernel, where a sorting algorithm can be applied, or the Laplacian kernel, which can be decomposed in a smooth part and a negative distance part, see Section~\ref{sec:energy} and \ref{sec:laplacian} for more details.
Lastly, it is worth noting that certain kernel methods may exhibit constraints in their ability to generalize effectively in very high-dimensions, irrespective of the computation of kernel sums.

\appendix

\section{Proof of Proposition~\ref{prop:sliced_kernels}}\label{proof:sliced_kernels}
The proof of the first part follows similar ideas as \cite[Thm 1]{HWAH2023} and \cite[Prop 2.2]{R2003}.
\\

\noindent
\textbf{Part (i):}
Since $f|_{[0,s]}\in L^{\infty}([0,s])$, we have that $f(s\cdot)|_{[0,1]}\in L^{\infty}([0,1])$.
Moreover, we have that $(1-t^2)^{\frac{d-3}{2}}\in L^1([0,1])$ for $d\geq 2$. Thus H\"olders formula yields that
$F(s)$ is finite for any $s\geq0$.

Let $s\geq0$ and $x\in\R^d$ with $\|x\|=s$. Denote by $U_x$ an orthogonal matrix such that $U_xx=\|x\|e_1$, where $e_1$ is the first unit vector.
Then, it holds
\begin{equation*}
P_\xi(x)=P_\xi(s U_x^\tT e_1)=s\langle \xi,U_x^\tT e_1\rangle=s\langle U_x\xi,e_1\rangle=sP_{U_x\xi}(e_1).
\end{equation*}
In particular, we have that 
\begin{align*}
\E_{\xi\sim \mathcal U_{\Sp^{d-1}}}[f(|P_\xi(x)|)]&=\E_{\xi\sim \mathcal U_{\Sp^{d-1}}}[f(s|P_{U_x\xi}(e_1)|)]
=\E_{\xi\sim \mathcal U_{\Sp^{d-1}}}[f(s|P_{\xi}(e_1)|)]=\E_{\xi\sim \mathcal U_{\Sp^{d-1}}}[f(s|\xi_1|)],
\end{align*}
where we used the substitution of $U_x\xi$ by $\xi$ in the second equality and the identity $P_\xi(e_1)=\langle \xi,e_1\rangle=\xi_1$ in the last one.
Using the notations $|\Sp^{d-1}|=\int_{\Sp^{d-1}} 1 \dx x$ for the volume of $\Sp^{d-1}$ and $\xi=\xi_1 e_1+\sqrt{1-\xi_1^2}(0,\xi_{2:d})$ with $\xi_{2:d}=(\xi_2,...,\xi_d)\in\Sp^{d-2}$ and applying \cite[eqt (1.16)]{AH2012}, we obtain 
\begin{align*}
\E_{\xi\sim \mathcal U_{\Sp^{d-1}}}[f(|P_\xi(x)|)]&=
\frac{1}{|\Sp^{d-1}|}\int_{\Sp^{d-1}}f(s|\xi_1|)\d\Sp^{d-1}(\xi)\\
&=
\frac{1}{|\Sp^{d-1}|}\int_{-1}^1\int_{\Sp^{d-2}}f(s|\xi_1|)\d\Sp^{d-2}(\xi_{2:d})(1-\xi_1^2)^{\frac{d-3}{2}}\d\xi_1\\
&=\frac{|\Sp^{d-2}|}{|\Sp^{d-1}|}\int_{-1}^1f(s|t|)(1-t^2)^{\frac{d-3}{2}}\d t.
\end{align*}
By  the symmetry of the integrand and the volume formula
$
|\Sp^{d-2}|=\frac{2\pi^{{(d-1)}/2}}{\Gamma(\frac{d-1}{2})}
$,
we conclude that
\begin{align*}
\E_{\xi\sim \mathcal U_{\Sp^{d-1}}}[f(|P_\xi(x)|)]=\frac{2\Gamma(\frac{d}{2})}{\sqrt{\pi}\Gamma(\frac{d-1}{2})}\int_0^1f(ts)(1-t^2)^{\frac{d-3}{2}}\d t\eqqcolon F(s).
\end{align*}
In particular, it holds
\begin{align*}
K(x,y)=\E_{\xi\sim \mathcal U_{\Sp^{d-1}}}[f(|P_\xi(x-y)|)]=F(\|x-y\|).
\end{align*}

\noindent
\textbf{Part (ii):} 
Using the substitution $r=ts$, we get
\begin{equation*}
F(s)=\frac{2\Gamma(\frac{d}{2})}{s\sqrt{\pi}\Gamma(\frac{d-1}{2})}\int_0^s f(r)\Big(1-\frac{r^2}{s^2}\Big)^{\frac{d-3}{2}}\d r
\end{equation*}
By the differentiability of $1/s$ on $(0,\infty)$ it remains to show the differentiability of $\int_0^s f(r)\Big(1-\frac{r^2}{s^2}\Big)^{\frac{d-3}{2}}\d r$. Since $d>3$, it holds for $r=1$ that $f(r)\Big(1-\frac{r^2}{s^2}\Big)^{\frac{d-3}{2}}=0$. Moreover, $\Big(1-\frac{r^2}{s^2}\Big)^{\frac{d-3}{2}}$ is differentiable with respect to $s$ with integrable derivative. Thus, Leibniz integral rule yields that $\int_0^s f(r)\Big(1-\frac{r^2}{s^2}\Big)^{\frac{d-3}{2}}\d r$ is differentiable with derivative
\begin{equation*}
\int_0^sf(r)\frac{\d}{\d s}\Big(1-\frac{r^2}{s^2}\Big)^{\frac{d-3}{2}} \d r.
\end{equation*}
This proves the claim.\hfill$\Box$

\section{Proof of Theorem~\ref{thm:kernels_to_sliced_analytic}}\label{proof:kernels_to_sliced_analytic}
We need two auxiliary lemmas.

\begin{lemma}\label{lem:power_series}
Let $a_n,b_n\in\R$, $n\in\N$ and $d\in\N$ such that
\begin{equation*}
b_n=\frac{\sqrt{\pi}\Gamma(\frac{n+d}{2})}{\Gamma(\frac{d}{2})\Gamma(\frac{n+1}{2})} a_n.
\end{equation*}
Then the power series $F(x)=\sum_{n=0}^\infty a_n x^n$ and $f(x)=\sum_{n=0}^\infty b_n x^n$ have the same convergence radius.
\end{lemma}
\begin{proof}
Denote by $r$ and $R$ the convergence radii of $f$ and $F$.
We aim to show that $r=R$.
First note that for $n>3$ we have due to the monotonicity of the gamma function and $\Gamma(z)z=\Gamma(z+1)$ that
\begin{equation*}
1\leq\frac{\Gamma(\frac{n+d}{2})}{\Gamma(\frac{n+1}{2})}= \frac{\Gamma(\frac{n+d}{2}-\lceil\frac{d}{2}\rceil)}{\Gamma(\frac{n+1}{2})}\prod_{k=1}^{\lceil\frac{d}{2}\rceil}\frac{n+d-2k}{2}\leq \Big(\frac{n+d}{2}\Big)^{\lceil\frac{d}{2}\rceil}
\end{equation*}
such that it there exist some $C_1,C_2>0$, independent of $n$, with
\begin{equation*}
C_1 |a_n|\leq |b_n|\leq C_1 \Big(\frac{n+d}{2}\Big)^{\lceil\frac{d}{2}\rceil} |a_n|\quad\text{for all }n\in\N.
\end{equation*}
Since power series are absolutely convergent in the interior of their convergence radius, this implies for all $|x|<r$  that
\begin{equation*}
\sum_{n=0}^\infty |a_n| |x|^n\leq \frac1{C_1}\sum_{n=0}^\infty |b_n| |x|^n<\infty
\end{equation*}
such that $F$  converges as well.
Therefore, we have that $R\geq r$.

Vice versa, let $0\neq |x|<R$ and denote $\epsilon=\sqrt{\frac{R}{|x|}}>1$. Then it holds that
\begin{equation*}
\sum_{n=0}^\infty |b_n| |x|^n\leq C_2\sum_{n=0}^\infty |a_n|\big(\frac{n+d}{2}\big)^{\lceil\frac{d}{2}\rceil}\epsilon^{-n} (\epsilon |x|)^n.
\end{equation*}
Since $\big(\frac{n+d}{2}\big)^{\lceil\frac{d}{2}\rceil}\epsilon^{-n}\to0$ as $n\to\infty$, it is bounded. Moreover, we have that $\epsilon |x|<R$. Using that $F$ is absolutely convergent for $|x|<R$, this yields that the right side of the above formula is finite. In particular, $f$ is convergent for $x$ and $r\geq R$.
\end{proof}

\begin{lemma}\label{lem:haessliche_integrale}
Let $r\in(-1,\infty)$ and $d\in\N$. Then it holds
\begin{equation*}
\int_0^1 t^r (1-t^2)^{\frac{d-3}{2}}\d t=\frac{\Gamma(\frac{d-1}{2})\Gamma(\frac{r+1}{2})}{2\Gamma(\frac{d+r}{2})}.
\end{equation*}
\end{lemma}
\begin{proof}
Using the substitution $s=t^2$ with $\frac{\d t}{\d s}=\frac12 s^{-\frac12}$, we obtain
\begin{align*}
\int_0^1 t^r (1-t^2)^{\frac{d-3}{2}}\d t&=\frac12\int_0^1s^{\frac{r}{2}}(1-s)^{\frac{d-3}{2}}s^{-\frac12}\d s\\
&=\frac12\int_0^1s^{\frac{r-1}{2}}(1-s)^{\frac{d-3}{2}}\d s=\Beta\Big(\frac{r+1}{2},\frac{d-1}{2}\Big)=\frac{\Gamma(\frac{d-1}{2})\Gamma(\frac{r+1}{2})}{2\Gamma(\frac{d+r}{2})},
\end{align*}
where $\Beta$ is the beta function.
\end{proof}

Now, we can provide the proof of our theorem.

\begin{proof}[Proof of  Theorem~\ref{thm:kernels_to_sliced_analytic}]
First, note that by Lemma~\ref{lem:power_series}, $F$ is globally convergent if and only if $f$ is globally convergent. 

Moreover, we obtain by Proposition~\ref{prop:sliced_kernels} that
\begin{align*}
F(s)=\frac{2\Gamma(\frac{d}{2})}{\sqrt{\pi}\Gamma(\frac{d-1}{2})}\int_0^1f(ts)(1-t^2)^{\frac{d-3}{2}}\d t,
\end{align*}
and using Lemma~\ref{lem:haessliche_integrale} further
\begin{equation*}
F(s) = \E_{\xi\sim \mathcal U_{\Sp^{d-1}}}[f(|P_\xi(x)|)]=\sum_{n=0}^\infty \frac{\Gamma(\frac{d}{2})\Gamma(\frac{n+1}{2})}{\sqrt{\pi}\Gamma(\frac{n+d}{2})}b_n\|x\|^n =\sum_{n=0}^\infty a_n\|x\|^n=F(\|x\|).
\end{equation*}
\end{proof}

\section{Additional Kernels}\label{app:laplace_hypergeom}

In this section, we provide additional information to the kernels in Table~\ref{tab:kernels}. More precisely, we express the function $f$ for the Mat\'ern kernel as hypergeometric function. This includes the Laplacian kernel which is the Mat\'ern kernel for $\nu=\frac12$ and $\beta=\frac{1}{\alpha}$.
Further, we derive the formula for the thin plate spline kernel \cite{D1975}.

\subsection{Mat\'ern and Laplacian Kernel}
The Mat\'ern kernel is defined for $\nu\not\in\Z$ via the basis function 
\begin{equation*}
F(x)=\frac{1}{\Gamma(\nu)2^{\nu-1}}\Big(\frac{\sqrt{2\nu}}{\beta}x\Big)^\nu K_\nu\Big(\frac{\sqrt{2\nu}}{\beta}x\Big),
\end{equation*}
where $K_\nu$ is the modified Bessel function of second kind. Moreover it can be continously extended to $\nu\in\Z$.
Recall that $K_\nu$ is defined by
\begin{equation*}
K_\nu(x)=\frac{\pi}{2}\frac{I_{-\nu}(x)-I_{\nu}(x)}{\sin(\nu\pi)},\quad I_\nu(x)=\sum_{n=0}^\infty\frac{x^{2n+\nu}}{n!\Gamma(n+\nu+1)2^{2n+\nu}}.
\end{equation*}
Thus, we have that
\begin{equation*}
F(x)=\frac{\pi}{\Gamma(\nu)2^{\nu}\sin(\nu\pi)}\Big(\sum_{n=0}^\infty\frac{\nu^{n}x^{2n}}{n!\Gamma(n-\nu+1)2^{n-\nu}\beta^{2n}}-\sum_{n=0}^\infty\frac{\nu^{n+\nu}x^{2n+2\nu}}{n!\Gamma(n+\nu+1)2^{n}\beta^{2n+2\nu}}\Big).
\end{equation*}
Using Proposition~\ref{prop:sliced_kernels} and Lemma~\ref{lem:haessliche_integrale} similarly as in the proof of Theorem~\ref{thm:kernels_to_sliced_analytic}, we obtain that the kernels $K(x,y)=F(\|x-y\|)$ and $\mathrm{k}(x,y)=f(|x-y|)$ fulfill \eqref{eq:sliced_kernel}, where
{\small
\begin{align*}
f(x)&=\frac{\pi^{3/2}}{\Gamma(\frac{d}{2})\Gamma(\nu)2^{\nu}\sin(\nu\pi)}\Big(\sum_{n=0}^\infty\frac{\Gamma(\frac{2n+d}{2})\nu^{n}x^{2n}}{\Gamma(\frac{2n+1}{2})n!\Gamma(n-\nu+1)2^{n-\nu}\beta^{2n}}\\
&\quad\phantom{\frac{\pi^{3/2}}{\Gamma(\frac{d}{2})\Gamma(\nu)2^{\nu}\sin(\nu\pi)}}-\sum_{n=0}^\infty\frac{\Gamma(\frac{2n+2\nu+d}{2})\nu^{n+\nu}x^{2n+2\nu}}{\Gamma(\frac{2n+2\nu+1}{2})n!\Gamma(n+\nu+1)2^{n}\beta^{2n+2\nu}}\Big)\\
&=\frac{\pi}{\Gamma(\nu)\Gamma(1-\nu)\sin(\nu\pi)}\sum_{n=0}^\infty\frac{\Gamma(n+\frac{d}{2})}{\Gamma(\frac{d}{2})}\frac{\Gamma(\frac12)}{\Gamma(n+\frac{1}{2})} \frac{\Gamma(1-\nu)}{\Gamma(n+1-\nu)} \frac{(\frac{\nu x^2}{2\beta^2})^n}{n!}\\
&\quad-\frac{\pi^{3/2}\Gamma(\frac{2\nu+d}{2})x^{2\nu}\nu^\nu}{\Gamma(\frac{d}{2})\Gamma(\nu)\Gamma(\nu+1)\Gamma(\nu+\frac{1}{2})2^{\nu}\sin(\nu\pi)\beta^{2\nu}}\sum_{n=0}^\infty\frac{\Gamma(n+\frac{2\nu+d}{2})}{\Gamma(\frac{2\nu+d}{2})}
\frac{\Gamma(\frac{2\nu+1}{2})}{\Gamma(n+\frac{2\nu+1}{2})}
\frac{\Gamma(\nu+1)}{\Gamma(n+\nu+1)}
\frac{(\frac{\nu x^2}{2\beta^2})^n}{n!}.
\end{align*}}
Using the definition of the hypergeometric function
\begin{equation*}
{_1}F_2(a;b,c;x)=\sum_{n=0}^\infty\frac{\Gamma(n+a)}{\Gamma(a)}\frac{\Gamma(b)}{\Gamma(n+b)}\frac{\Gamma(c)}{\Gamma(n+c)}\frac{x^n}{n!}
\end{equation*}
and the identities $\Gamma(\nu)\Gamma(1-\nu)=\frac{\pi}{\sin(\nu\pi)}$ and $\Gamma(\nu+\frac12)\Gamma(\nu+1)=2^{-2\nu}\pi^{\frac12}\Gamma(2\nu+1)$, we obtain that
\begin{align*}
f(x)&={_1}F_2(\tfrac{d}{2};\tfrac12,1-\nu;\tfrac{\nu x^2}{2\beta^2})-\frac{\Gamma(1-\nu)\Gamma(\nu+\frac{d}{2})x^{2\nu}(2\nu)^\nu}{\Gamma(\frac{d}{2})\Gamma(2\nu+1)\beta^{2\nu}}{_1}F_2(\nu+\tfrac{d}2;\nu+\tfrac12,\nu+1;\tfrac{\nu x^2}{2\beta^2}).
\end{align*}

\paragraph{Special values of $\nu$}
For evaluating the Mat\'ern kernel with $\nu=p+\frac12$, we implement the Mat\'ern kernel via the identity
\begin{equation*}
F(x)=\exp\Big(-\frac{\sqrt{2p+1}x}{\beta}\Big)\frac{p!}{(2p)!}\sum_{n=0}^p\frac{(p+n)!}{n!(p-n)!}\Big(\frac{2\sqrt{2p+1}x}{\beta}\Big)^{p-n}.
\end{equation*}
For $\nu=\frac12$, we obtain the Laplacian kernel with $\alpha=\frac1\beta$. In this case, the function $f$ can be expressed as
\begin{align*}
f(x)
&={_1}F_2(\tfrac{d}{2};\tfrac12,\tfrac12;\tfrac{\alpha^2x^2}{4})-\frac{\sqrt{\pi}\alpha x\Gamma(\frac{d+1}{2})}{\Gamma(\frac{d}{2})}{_1}F_2(\tfrac{d+1}{2};1,\tfrac32;\tfrac{\alpha^2x^2}{4})
\end{align*}

\subsection{Thin Plate Spline Kernel}

Let $\mathrm{k}(x,y)=f(|x-y|)$ with $f(x)=d x^2\log(x)-C x^2$ where
\begin{equation}\label{eq:thin_pline_C}
C\coloneqq\frac{4d\Gamma(\frac{d+2}{2})}{\sqrt{\pi}\Gamma(\frac{d-1}{2})}\int_0^1t^2\log(t)(1-t^2)^{\frac{d-3}{2}} \d t=-\frac{d}{2}\big(H_\frac{d}{2}-2+\log(4)\big),
\end{equation}
where $H_x=\int_0^1\frac{1-t^x}{1-t}\d t$ is the harmonic number and where the equality follows from the integral identity \cite[4.253.1]{GR2007} which reads as
\begin{equation*}
\int_0^1t^2\log(t)(1-t^2)^{\frac{d-3}{2}} \d t=B(\tfrac{3}{2},\tfrac{d-1}{2})(\psi(\tfrac32)-\psi(\tfrac{d-2}{2}))=-\frac{\sqrt{\pi}\Gamma(\frac{d-1}{2})(H_\frac{d}{2}-2+\log(4))}{8\Gamma(\frac{d+2}{2})}.
\end{equation*}
Then we have by Proposition~\ref{prop:sliced_kernels} that $K$ given by \eqref{eq:sliced_kernel} is given by $K(x,y)=F(\|x-y\|)$ with
\begin{equation*}
F(s)=\frac{2\Gamma(\frac{d}{2})}{\sqrt{\pi}\Gamma(\frac{d-1}{2})}\int_0^1\big(d s^2t^2\log(st)-Cs^2t^2\big)(1-t^2)^{\frac{d-3}{2}}\d t
\end{equation*}
Using the linearity of the integral and $\log(st)=\log(s)+\log(t)$, this is equal to
\begin{align*}
&\frac{2d\Gamma(\frac{d}{2})}{\sqrt{\pi}\Gamma(\frac{d-1}{2})}\Big(s^2\log(s)\int_0^1t^2(1-t^2)^{\frac{d-3}{2}}\d t+s^2\int_0^1t^2\log(t)(1-t^2)^{\frac{d-3}{2}}\d t\Big)\\
&-Cs^2\frac{2\Gamma(\frac{d}{2})}{\sqrt{\pi}\Gamma(\frac{d-1}{2})}\int_0^1t^2(1-t^2)^{\frac{d-3}{2}}\d t
\end{align*}
By inserting $C$, the formula from Lemma~\ref{lem:haessliche_integrale} for the integrals and using $\frac{\Gamma(\frac{d+2}{2})}{\Gamma(\frac{d}{2})}=\frac{d}{2}$ this is equal to
\begin{equation*}
s^2\log(s)+s^2 \frac{C}{d} - s^2\frac{C}{d}=s^2\log(s).
\end{equation*}
Thus, we can summarize $F(s)=s^2\log(s)$ such that $K$ is the thin plate spline kernel.

\subsection*{Acknowledgements}
We would like to thank Robert Gruhlke, Tim Jahn, Aleksei Kroshnin, Gabriele Steidl, Michael Quellmalz and Christian Wald for fruitful discussions.
Additionally, we would like to thank the anonymous reviewers for their valuable comments.

Funding within the EPSRC programme grant ``The Mathematics of Deep Learning'' with reference EP/V026259/1 is gratefully acknowledged. 
For the purpose of open access, the author has applied a Creative Commons Attribution (CC BY) licence to any Author Accepted Manuscript version arising.

\bibliographystyle{abbrv}
\bibliography{ref}

\end{document}